\documentclass[leqno,12pt]{amsart}
\usepackage{amssymb,amsmath,amsfonts,amsthm,mathtools,latexsym,ifthen,bezier}
\usepackage{graphicx}
\usepackage{amsxtra}
\usepackage{xspace}
\usepackage{dirtytalk}
\usepackage{url}
\usepackage[english]{babel} 

\usepackage{enumitem}
\usepackage[margin=1.3in]{geometry}
\usepackage{setspace}
\usepackage[english]{babel}
\usepackage{latexsym}
\usepackage{enumitem}
\usepackage{amsthm}
\usepackage{amssymb}
\DeclareMathAlphabet{\mathpzc}{OT1}{pzc}{m}{it}
\usepackage[latin1]{inputenc}
\usepackage{afterpage}

\usepackage{bbm}
\usepackage[rgb,dvipsnames]{xcolor}
\usepackage{tikz} 
\usetikzlibrary{decorations.text}
\usetikzlibrary{arrows,arrows.meta,hobby}
\usetikzlibrary{shapes.misc, positioning}

\usepackage{setspace}
\usepackage[english]{babel}
\usepackage{bussproofs}
\usepackage{csquotes}
\usepackage{dirtytalk}
\usepackage{mathtools}
\usepackage{mathrsfs}
\usepackage{latexsym}
\usepackage{enumitem}
\usepackage{amsthm}
\usepackage{amssymb}
\DeclareMathAlphabet{\mathpzc}{OT1}{pzc}{m}{it}
\usepackage[latin1]{inputenc}
\usepackage{afterpage}

\newtheorem{maintheorem}{Main Theorem}
\newtheorem{theorem}{Theorem}[section]

\newtheorem{lemma}[theorem]{Lemma}
\newtheorem{proposition}[theorem]{Proposition}
\newtheorem{corollary}[theorem]{Corollary}
\newtheorem{observation}[theorem]{Observation}
\newtheorem{fact}[theorem]{Fact}

\theoremstyle{definition}
\newtheorem{definition}[theorem]{Definition}
\newtheorem{example}[theorem]{Example}
\theoremstyle{remark}

\newtheorem{question}{Question}

\begin{document}

\title{The Cicho\'n Diagram for Degrees of Relative Constructibility}

\author[Switzer]{Corey Bacal Switzer}

\address[C.~Switzer]{Mathematics, The Graduate Center of The City University of New York, 365 Fifth Avenue, New York, NY 10016}

\email{cswitzer@gradcenter.cuny.edu}

\date{}

\maketitle

\begin{abstract}
\noindent Following a line of research initiated in \cite{BBNN}, I describe a general framework for turning reduction concepts of relative computability into diagrams forming an analogy with the Cicho\'n diagram for cardinal characteristics of the continuum. I show that working from relatively modest assumptions about a notion of reduction, one can construct a robust version of such a diagram. As an application, I define and investigate the Cicho\'n Diagram for degrees of constructibility relative to a fixed inner model $W$. Many analogies hold with the classical theory as well as some surprising differences. Along the way I introduce a new axiom stating, roughly, that the constructibility diagram is as complex as possible\footnote{This reasearch was supported by a CUNY mathematics fellowship and the author would like to thank the mathematics department at the Graduate Center at CUNY for this. The author would also like to express his gratitude to Professor Joel David Hamkins for his patient, enthusiastic and thoughtful help and encouragement as well as Professors Gunter Fuchs and Alfred Dolich for sitting on the author's oral exam committee where a version of this material was originally presented, Professor J\"org Brendle for very helpful comments on an earlier draft, and the anonymous referee for drawing his attention to \cite{Kihara17} and \cite{GreenbergTuretskyKuyper}.

}.

 \end{abstract}

\section{Introduction}

Building off of an original idea of Rupprecht, \cite{Rupprecht}, in \cite{BBNN} an analogue of the Cicho\'n diagram was developed for highness properties of Turing degrees. This idea has blossomed into a growing field of ``effective cardinal characteristics", see \cite{Kihara17} and the survey \cite{GreenbergTuretskyKuyper}. The framework set up in \cite{BBNN} is very flexible and can be used to produce a wide variety of Cicho\'n Diagrams for various reductions related to various notions of computability, see \cite{Kihara17} and Section 5 of \cite{GreenbergTuretskyKuyper}. Expanding upon this more general viewpoint, I show in this article that such diagrams exist for many of the standard reduction concepts on the reals. In each case I obtain an analogue of (a large fragment of) the Cicho\'n diagram. As an example, I show how such a diagram can be constructed and studied for degrees of constructibility relative to some inner model $W$ alongside the corresponding reduction $\leq_W$.
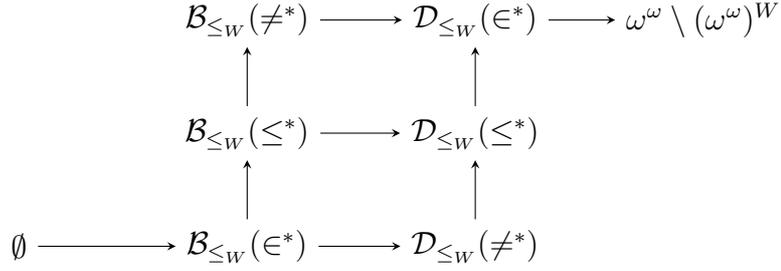
\begin{figure}[h]\label{Figure.Cichon-basic}
\centering
  \begin{tikzpicture}[scale=1.5,xscale=2]
     \draw (0,0) node (empty) {$\emptyset$}
           (1,0) node (Bin*) {$\mathcal B_{\leq_W}(\in^*)$}
           (1,1) node (Bleq*) {$\mathcal B_{\leq_W}(\leq^*)$}
           (1,2) node (Bneq*) {$\mathcal B_{\leq_W}(\neq^*)$}
           (2,0) node (Dneq*) {$\mathcal D_{\leq_W}(\neq^*)$}
           (2,1) node (Dleq*) {$\mathcal D_{\leq_W}(\leq^*)$}
           (2,2) node (Din*) {$\mathcal D_{\leq_W}(\in^*)$}
           (3,2) node (all) {$\omega^\omega\setminus(\omega^\omega)^W$}
           ;
     \draw[->,>=stealth]
            (empty) edge (Bin*)
            (Bin*) edge (Bleq*)
            (Bleq*) edge (Bneq*)
            (Bin*) edge (Dneq*)
            (Bleq*) edge (Dleq*)
            (Bneq*) edge (Din*)
            (Dneq*) edge (Dleq*)
            (Dleq*) edge (Din*)
            (Din*) edge (all)
            ;
  \end{tikzpicture}
  \caption{The Cicho\'n diagram for $\leq_W$}
\end{figure}

Let $W$ be a transitive inner model of ZFC. Recall that for reals $x$ and $y$ (in $V$) the relation $x \leq_W y$ is defined by $x \in W[y]$. The following theorem, which I formalize in $\mathsf{GBC}$, is the first of two main theorems I prove. 
\begin{maintheorem}
For any transitive inner model $W$ of ZFC, the inclusions shown in the Cicho\'n diagram for $\leq_W$ (Figure 1) all hold. Furthermore, the diagram is complete in the sense that there is a forcing extension showing that no other implications are necessarily true.
\end{maintheorem}
In fact I show more: for a wide variety of computability-like notions alongside their corresponding reduction concepts one can construct a Cicho\'n diagram similar to the one pictured above. In the case of $\leq_W$ the theory and corresponding diagram are robust in that they interact very well with regards to the familiar forcings to add reals that one studies for the classical Cicho\'n diagram. Moreover, in many cases the constructibility perspective makes clear distinctions in the diagram that are lost when only considering the cardinals. Indeed several well known forcing notions naturally split the diagram into several pieces whereas its iteration may only produce two cardinals. Finally I show that by a simple forcing over the inner model $W$ the diagram can be saturated in the sense that all possible separations can be realized simultaneously and this can be done in such a way that is indestructible with respect to further forcing. This is the second main theorem of this paper.
\begin{maintheorem}
There is a model of ZFC realizing simultaneously all possible separations between nodes of the Cicho\'n diagram for $\leq_W$. Moreover, no further forcing over this model can destroy this property.
\end{maintheorem}

I dub the statement that \say{all possible separations of the $\leq_W$ Cicho\'n diagram are realized}, $CD(\leq_W)$. The paper finishes by briefly treating $CD(\leq_W)$ as an axiom.

The article \cite{Kihara17} is an important predecessor to this work and contains many related ideas. In that work, the author considers develops a very fine analysis of various notions from arboreal forcing theory that translates well to the context of Spector pointclasses. As a result, the same connections found in \cite{BBNN} between cardinal characteristics and Turing degrees are explored at the level of hyperarithmetic degrees (and beyond) and further analogues between higher randomness and set theory of the reals are considered. The main results show that at these levels of definability the Cicho\'n diagram can look different than in either the $\leq_T$ context or the set theory context. The essential difference between the current paper and \cite{Kihara17} is that instead of importing ideas from set theory to the computability side, here I go the other way. In particular my focus on degrees of relative constructibility implies a much coarser notion of definability which is inherently more set theoretic than hyperarithmetic degrees. Consequently the results line up more closely with those from the classical theory of set theory of the reals. However the new perspective offers a more fine-grained analysis of which types of reals are added by which forcing notions, thus providing a different look at the Cicho\'n digram. This is elaborated on in detail in Sections 3 and 4. While many of these results are well known, the new perspective ilucidates further the use on the set theory side of the computational properties of single reals, a central idea in both finite and higher computability. In some cases moreover, particularly for the $\mathbb{LOC}$ forcing, the results presented here appear to be new.


\section{Generalized Cicho\'n Diagrams for Reductions}
In this section I expand on the general viewpoint of reduction concepts as giving rise to Cicho\'n diagrams. A similar idea is explored in Section 5 of \cite{GreenbergTuretskyKuyper} and in \cite{Kihara17}\footnote{Thanks to the anonymous referee for pointing this out to me.}. In contrast with those papers though I work more on the level building analogues of the Cicho\'n diagram than in considering relations between various types of reducibilities in higher computability theory and Tukey reductions. In particular, Theorem \ref{propdi}, which in slightly less general contexts is essentially folklore, shows that for any ``reasonable" reduction concept, a corresponding Cicho\'n diagram exists. 

Underlying the construction of Cicho\'n diagrams for reduction concepts is a certain perspective on cardinal characteristics of the continuum. To describe this perspective better, let us think of cardinal characteristics of the continuum in terms of small and large sets relative to some relation giving this notion of smallness and largeness. For example, recall that the binary relation $\leq^*$ is defined on $\omega^\omega$ as $f \leq^* g$ if and only if for all but finitely many $n \in \omega$, $f(n) \leq g(n)$. A family reals $A$ is ($\leq^*)$ -{\em unbounded} if for all $f \in \omega^\omega$ there is some $g \in A$ such that $g \nleq^* f$. The smallest cardinality of an unbounded family is called the unbounding number, denoted $\mathfrak b = \mathfrak b(\leq^*)$. Dually, a family of reals $A \subseteq \omega^\omega$ is ($\leq^*$) -{\em dominating} if for all $g \in \omega^\omega$ there is a $f \in A$ such that $g \leq^* f$. The least size of a dominating family is called the dominating number, denoted $\mathfrak{d} = \mathfrak{d}(\leq^*)$. Intuitively one thinks of bounded families as being \say{small} and dominating families as being \say{big}. Thus, heuristically one might think of $\mathfrak{b}$ as the least size of a set that's not \say{small} and $\mathfrak{d}$ as the least size of a set that's \say{big}. To obtain an analogy in the computable world, the authors of \cite{BBNN} define $\mathcal B (\leq^*)$ as the set of oracles computing a function $f$ such that $g \leq^* f$ for each computable function $g$ and $\mathcal D (\leq^*)$ as the set of oracles computing a function $f$ such that $f \nleq^* g$ for all computable $g$. In other words $\mathcal B(\leq^*)$ is the set of oracles which can compute a witness to the fact that the computable functions are \say{small} and $\mathcal D(\leq^*)$ is the set of oracles which can compute a witness to the fact that the computable functions are not \say{big}. Moreover, these sets turn out to correspond to \say{highness} properties of Turing degrees that are well studied in computability theory. Specifically, by a theorem of Martin (cf \cite[pp. 3]{BBNN}), $\mathcal B(\leq^*)$ is the set of high degrees and, by definition, $\mathcal D(\leq^*)$ is the set of hyperimmune degrees.

This formalism has nothing to do with {\em Turing} computability per se. This motivates the following general definition.

\begin{definition}
A {\em reduction concept} is a triple $(X, \sqsubseteq, x_0)$ where $X$ is a nonempty set, $x_0 \in X$ is some distinguished element and $\sqsubseteq$ is a partial pre-order on $X$. If $X$ is given or implicit, we also say that the pair $(\sqsubseteq , x_0)$ is a {\em reduction concept on $X$}. If $(X, \sqsubseteq, x_0)$ is a reduction concept, then for $x, y \in X$ say that $x$ is $\sqsubseteq$-{\em reducible to} $y$ if $x \sqsubseteq y$ and say that $x$ is $\sqsubseteq$-{\em basic} if it is $\sqsubseteq$-reducible to $x_0$. 

Let $(\sqsubseteq, x_0)$ be a reduction concept on $X$ and $R \subseteq X \times X$ be a binary relation. Let $\sqsubseteq \upharpoonright x_0 = \{y\in X \; | \; y \sqsubseteq x_0\}$ be the basic reals. Then define the {\em bounding set} for $R$ as 

\[ \mathcal B_\sqsubseteq (R) = \{x \in X \; | \; \exists y \sqsubseteq x \; \forall z \in \sqsubseteq \upharpoonright x_0 \; [zRy]\}\]

\noindent and the {\em non-dominating set} for $R$ as 

\[ \mathcal D_\sqsubseteq (R) = \{x \in X \; | \; \exists y \sqsubseteq x \; \forall z \in \sqsubseteq \upharpoonright x_0 \; [\neg yRz]\}.\] 
\end{definition}
Roughly, if $\sqsubseteq$ is some sort of relative computability relation, then $\mathcal B_\sqsubseteq (R)$ is the set of elements of $x \in X$ which compute an $R$-bound on the computable elements of $X$ and $\mathcal D_\sqsubseteq(R)$ is the set of $x \in X$ which compute an element which is not $R$-dominated by the set of all computable elements. If $R$ is a relation giving a notion of \say{small} and \say{big} sets as described above one can think of $\mathcal B_\sqsubseteq (R)$ as the set of elements computing a witness to the fact that the $\sqsubseteq$-basic sets are small and $\mathcal D_\sqsubseteq (R)$ as the set of elements computing a witness to the fact that the $\sqsubseteq$-basic elements are not big.

\begin{example}[\cite{BBNN}]
Let $x_0 \in \omega^\omega$ be some computable real, say the constant function at $x_0$. Then the pair $(\leq_T, x_0)$ forms a reduction concept on the reals. The basic reals are the computable reals. For any binary relation $R$ on the reals $\mathcal B_{\leq_T}(R)$ is the set of Turing degrees computing an element of $X$ which $R$-bounds all the computable sets. Similarly $\mathcal D_{\leq_T} (R)$ is the set of Turing degrees computing an element of $X$ which is not $R$-dominated by any computable set.
\end{example}

The next example will be the central focus of the rest of this article.
\begin{example}
Let $x_0 \in \omega^\omega$ be constructible. Then the pair $(\leq_L, x_0)$ is a reduction concept on $\omega^\omega$ where $x \leq_L y$ if $x \in L[y]$. The basic reals are the constructible reals. More generally, fix some inner model $W \subseteq V$ and let $\leq_W$ be constructibility relative to $W$. Then if $0 \in (\omega^\omega)^W$ is any given real in $W$ the pair $(\leq_W, x_0)$ forms a reduction concept on Baire space and the basic reals are those of $W$. Since this is the main case let me explicit what the bounding and non-dominating sets are. Let $R$ be a relation on the reals of $V$. The set $\mathcal B_{\leq_W} (R)$ consists of all reals $x$ in $V$ such that in $W[x]$ there is an $R$-bound on the reals of $W$. Similarly the set $\mathcal D_{\leq_W} (R)$ consists of all reals $x$ in $V$ such that in $W[x]$ there is a real which is not $R$-bounded by any real in $W$. For example, $\mathcal B_{\leq_W}(\leq^*)$ is the set of dominating reals over $W$ in $V$ and $\mathcal D_{\leq_W}(\leq^*)$ is the set of unbounded reals over $W$ in $V$.
\end{example}
I will come back to this example in the next section. First, let me give some more examples of reduction concepts on the reals, though I will not treat them in detail in this article.


\begin{example}
Recall that the relation of many-one polytime reduction, $\leq_m^p$ is defined by $x \leq_m^p y$ if and only if there is a function $f$ which is computable in polynomial time such that $n \in x$ if and only if $f(n) \in y$. The pair $(\leq_M^p, \emptyset)$ is a reduction concept on $\mathcal P(\mathbb N)$.
\end{example}

\begin{example}
Let $\kappa > \omega$ be an uncountable cardinal. Recently there has been much work in the descriptive set theory of \say{generalized} Baire and Cantor spaces, $\kappa^\kappa$ and $2^\kappa$, including various generalizations of cardinal characteristics of the continuum. The same can be done in my framework for degrees of constructibility. For instance notions of eventual domination, etc all make sense in the general context of $\kappa^\kappa$ and corresponding bounding and non-dominating sets can be constructed over the basic elements, $(\kappa^\kappa)^L$.
\end{example}

The framework described above is flexible enough that $(X, \sqsubseteq, x_0)$ need not be some actual notion of computability on the reals nor have an explicit relation to cardinal characteristics of the continuum. For instance one might consider a class of models of a fixed theory in a fixed language with embeddibility. In this case, depending on the relations $R$ one studied, one would arrive at a diagram corresponding to when models with certain properties embed into one another. There are many possibilities, each giving a potentially interesting diagram of inclusions between the various bounding and non-dominating sets for an appropriate collection of relations. In future work I hope to explore all of these more fully. 

Presently however, let me restrict my attention to the types of cases described in the preceding examples. To see how these examples can lead to \say{Cicho\'n diagrams} let me define some relations.

\begin{definition}[Combinatorial relations]
I consider the reals as elements of Baire space, $\omega^\omega$. Let $f, g$ be reals. Then
\begin{enumerate}
\item
 $f \neq^* g$ if there is some $k$ such that for all $l > k$ $f(l) \neq g(l)$. In this case say that $g$ is {\em eventually not equal to} $f$. Note that the negation of $\neq^*$ is {\em infinitely often equal}, not eventual equality. 
 \item
 Let $h \in \omega^\omega$ and recall that an $h$-{\em slalom} is a function $\sigma: \omega \to [\omega]^{< \omega}$ such that for all $n \in \omega$ the set $|\sigma(n)| \leq h(n)$. In the case where $h$ is the identity function call $\sigma$ simply a {\em slalom}. For a slalom $\sigma$, I write $f \in^* \sigma$ if there is some $k$ such that for all $l > k$ $f(l) \in \sigma(l)$. In this case say that $f$ {\em is eventually captured by} $\sigma$. 
 \end{enumerate}
\end{definition}

Even in this general framework I can now prove a collection of implications giving a version of the Cicho\'n diagram.
\begin{theorem}
Let $(\sqsubseteq, x_0)$ be a reduction concept on $\omega^\omega$ extending $\leq_T$ such that if $f, g \sqsubseteq h$ then $f \circ g \sqsubseteq h$ then, interpreting arrows as inclusions, the following all hold: 

\begin{figure}[h]\label{Figure.Cichon-basic}
\centering
  \begin{tikzpicture}[scale=1.5,xscale=2]
     \draw (0,0) node (empty) {$\emptyset$}
           (1,0) node (Bin*) {$\mathcal B_\sqsubseteq(\in^*)$}
           (1,1) node (Bleq*) {$\mathcal B_\sqsubseteq(\leq^*)$}
           (1,2) node (Bneq*) {$\mathcal B_\sqsubseteq(\neq^*)$}
           (2,0) node (Dneq*) {$\mathcal D_\sqsubseteq(\neq^*)$}
           (2,1) node (Dleq*) {$\mathcal D_\sqsubseteq(\leq^*)$}
           (2,2) node (Din*) {$\mathcal D_\sqsubseteq(\in^*)$}
           (3,2) node (all) {$\omega^\omega\setminus \{x \; | \; x \sqsubseteq 0\}$}
           ;
     \draw[->,>=stealth]
            (empty) edge (Bin*)
            (Bin*) edge (Bleq*)
            (Bleq*) edge (Bneq*)
            (Bleq*) edge (Dleq*)
            (Dneq*) edge (Dleq*)
            (Bin*) edge (Dneq*)
            (Bneq*) edge (Din*)
            (Dleq*) edge (Din*)
            (Din*) edge (all)
            
            ;
  \end{tikzpicture}
  \caption{A Cicho\'n diagram for an arbitrary reduction concept on Baire space}
\end{figure}
\label{propdi}
\end{theorem}

\begin{proof}
Note that slaloms can be computably coded by reals so, since the relation $\sqsubseteq$ extends Turing computability the $\in^*$ can be seen as a relation on the reals. I drop the $\sqsubseteq$ subscript for readability. Also, I'll write \say{basic} for $\sqsubseteq$-basic and if $y \sqsubseteq x$ then I'll say that \say{$x$ builds $y$}. The requirement that $\sqsubseteq$ be closed downwards under compositions will be used implicitly throughout the argument where I will show that a function can build two other functions hence it can build their composition.

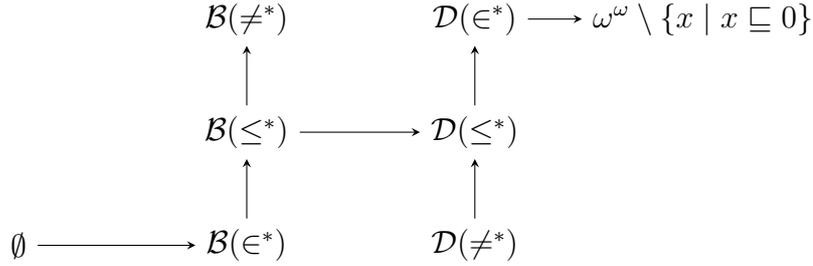
\begin{figure}[h]\label{Figure.Cichon-basic}
\centering
  \begin{tikzpicture}[scale=1.5,xscale=2]
     \draw (0,0) node (empty) {$\emptyset$}
           (1,0) node (Bin*) {$\mathcal B(\in^*)$}
           (1,1) node (Bleq*) {$\mathcal B(\leq^*)$}
           (1,2) node (Bneq*) {$\mathcal B(\neq^*)$}
           (2,0) node (Dneq*) {$\mathcal D(\neq^*)$}
           (2,1) node (Dleq*) {$\mathcal D(\leq^*)$}
           (2,2) node (Din*) {$\mathcal D(\in^*)$}
           (3,2) node (all) {$\omega^\omega \setminus \{x \; | \; x \sqsubseteq 0\}$}
           ;
     \draw[->,>=stealth]
            (empty) edge (Bin*)
            (Bin*) edge (Bleq*)
            (Bleq*) edge (Bneq*)
            (Bleq*) edge (Dleq*)
            (Dneq*) edge (Dleq*)
            (Dleq*) edge (Din*)
            (Din*) edge (all)
            ;
  \end{tikzpicture}
  \caption{The easy cases}
\end{figure}

All but two cases are essentially immediate from the definitions. These easy cases are pictured in Figure 3.For example, consider $\mathcal B(\in^*) \subseteq \mathcal B (\leq^*)$. This says that every $x$ building a slalom eventually capturing all the basic reals builds a real which eventually dominates all basic reals. This is proved as follows. Suppose $x \in \mathcal B(\in^*)$ and let $\sigma \sqsubseteq x$ be a slalom witnessing this. Then, define $z(n) = {\rm max} \, (\sigma (n)) +1$. Notice that $z \leq_T \sigma$ so $z \sqsubseteq \sigma$ and hence $z \sqsubseteq x$. Moreover, since $\sigma$ eventually captures all basic reals, $z$ must eventually dominate them all so $x \in \mathcal B (\leq^*)$. The other easy cases are similarly shown.

The two more substantive inclusions are $\mathcal B(\in^*) \subseteq\mathcal D(\neq^*)$ and $\mathcal B(\neq^*) \subseteq \mathcal D(\in^*)$. Let's start with $\mathcal B(\in^*) \subseteq \mathcal D(\neq^*)$. Substantively this states that if a real $x$ builds a slalom eventually capturing all basic functions then $x$ also builds a real which is infinitely-often-equal to all basic functions. In fact I will show a more general claim that implies this. The following lemma and proof is essentially a reinterpretation of Theorem 1.5 from \cite{Bar1987}.

\begin{lemma}
For any real $x$ the following are equivalent. 
\begin{enumerate}
\item
There is a real $g \sqsubseteq x$ such that for all basic $f \in \omega^\omega$, there exist infinitely many $n \in \omega$ such that $g(n) = f(n)$
\item
There is a basic $h \in \omega^\omega$ and an $h$-slalom $\sigma \sqsubseteq x$ such that for all basic $f \in \omega^\omega$ there are infinitely many $n \in \omega$ such that $f(n) \in \sigma(n)$.
\end{enumerate}
Moreover, given an infinitely-often-equal real as in $1)$, one can build from it an $h$-slalom as in $2)$ and given an $h$-slalom $\sigma$ as in $2)$ one can build an infinitely-often-equal real as in $1)$. Thus, $x \in \mathcal D(\neq^*)$ if and only if there is a basic $h \in \omega^\omega$ and an $h$-slalom which captures each of the basic reals infinitely often.
\label{infslalom}
\end{lemma}
Before proving Lemma \ref{infslalom}, notice that it implies the inclusion $\mathcal B(\in^*) \subseteq \mathcal D(\neq^*)$ since any slalom which captures every basic real cofinitely often must in particular capture each basic real infinitely often so if $x \in \mathcal B(\in^*)$ builds such a slalom, by the lemma $x$ must be able to build an infinitely-often-equal real as well.

\begin{proof}[Proof of Lemma \ref{infslalom}]
The forward direction is obvious: suppose that $g$ is an infinitely-often-equal real. Then clearly the $1$-slalom $\phi:\omega \to [\omega]^1$ such that $\phi(n) = \{g(n)\}$ is $\leq_T$-computable from $g$ and hence $\sqsubseteq$-reducible to $g$, thus giving the desired $h$-slalom.

For the backward direction fix a basic real $h$ such that there exists an $h$-slalom as in the statement of 2. I need to find a real $g$ which is infinitely often equal to every basic real. In a basic fashion, fix a family of finite, nonempty, pairwise disjoint subsets of $\omega$ enumerated $\{J_{n, k} \; | \; n < \omega \; \& \; k \leq h(n)\}$ which collectively cover $\omega$. Since $h$ is assumed to be basic there is no problem building such a partition, for example one could use singletons. Label $J_n = \bigcup_{k \leq h(n)} J_{n, k}$. Then for each basic $f \in \omega^\omega$ let $f':\omega \to \omega^{< \omega}$ be the function defined by $f'(n) = f\upharpoonright J_n$. More generally let $\mathcal J = \{f: \omega \to \omega^{<\omega} \; | \; {\rm dom} ( f(n) ) = J_n\}$. Notice that the basic elements of $\mathcal J$ are exactly $\{f' \;|\; f \in \omega^\omega \; \& \; f \sqsubseteq 0\}$ since from any $f'$ we can build $f$ and vice versa (by the the fact that the $J_n$'s are basic). But now since the $f'$'s are basic and each one codes a real one can by applying 2 plus some simple coding to find an $h$-slalom, $\sigma: \omega \to (\omega^{< \omega})^{< \omega}$ such that for every $n \in \omega$ $|\sigma(n)| \leq h(n)$ and $\sigma(n)$ is a set of finite partial functions from $J_n$ to $\omega$ and for every basic $f' \in \mathcal J$ there are infinitely many $n\in \omega$ such that $f'(n) \in \sigma(n)$. 

Let me denote $\sigma(n) = \{w^n_1,...,w^n_{h(n)}\}$. Now set $g_n = \bigcup_{k \leq h(n)} w^n_k \upharpoonright J_{n, k}$ and let $g = \bigcup_{n < \omega} g_n$. Notice that this gives an element of $\omega^\omega$ since the $J_{n, k}$'s were disjoint and collectively covered $\omega$. I claim that $g$ is as needed. Clearly $g$ is reducible to the $J_{n, k}$'s, which are basic, and the $w^n_k$'s, which are reducible to $\sigma$ so $g$ is reducible to $\sigma$. It remains to see that it is an infinitely-often-equal real. To see this, let $f \in \omega^\omega$ be basic and fix some $n$ such that $f'(n) \in \phi(n)$ (recall that there are infinitely many such $n$). Notice that since $f'(n) \in \phi(n)$ there must be some $k \leq h(n)$ such that $f\upharpoonright J_m = w^n_k$. Now let $x_n \in J_{n, k}$ (recall that this set is assumed to be non-empty). We have that $f(x_n) = w^n_k(x_n) = g(x_n)$. But there are infinitely many such $n$ and hence infinitely many such $x_n$ so this completes the proof.
\end{proof}

A similar proof produces the last inclusion, $\mathcal B(\neq^*) \subseteq \mathcal D(\in^*)$. In words this inclusion states that any real which can build a real which is eventually different from all basic reals can build a real which is not eventually captured by any given slalom. I will prove the following more general lemma, whose statement and proof is inspired by \cite{Bar1987}, Theorem 2.2. Given an $h$-slalom $\sigma$ and a function $f$ let me say that $f$ is {\em eventually never captured by} $\sigma$ if there is some $k$ such that for all $l > k$ $f(l) \notin \sigma (l)$.
\begin{lemma}
For any real $f$, the following are equivalent.
\begin{enumerate}
\item
The real $f$ is eventually different from all basic reals.
\item
The real $f$ is such that for all basic reals $h$ and all basic $h$-slaloms $\sigma$ for all but finitely many $n \in \omega$ $f(n) \notin \sigma (n)$.
\end{enumerate}
Therefore $x \in \mathcal B(\neq^*)$ if and only if $x$ builds a real which is eventually never captured by any basic $h$-slalom for any basic $h$.
\label{neqinf}
\end{lemma}
Let me note before I prove Lemma \ref{neqinf} that it proves the inclusion $\mathcal B(\neq^*) \subseteq \mathcal D(\in^*)$ and hence Theorem \ref{propdi}. To see why, suppose that $x \in \mathcal B(\neq^*)$ and, without loss of generality suppose that $x$ itself is a real which is eventually different from all basic reals. Then by the lemma $x$ is eventually never captured by any basic slalom so, in particular for infinitely many $n$ $x(n) \notin \sigma (n)$ for all basic $\sigma$, which means $x \in \mathcal D(\in^*)$.

\begin{proof}[Proof of Lemma \ref{neqinf}]
Fix some $f \in \omega^\omega$. The backward direction of this lemma is easy: if $f$ is eventually never captured by any basic $h$-slalom for any basic $h$ then in particular it is eventually never captured by the slalom sending $n \mapsto \{g(n)\}$ for each basic $g$ and hence it is eventually different from each basic $g$.

For the forward direction, assume $f$ is eventually different from all basic functions. Fix a basic $h$ and, like in the proof of Lemma \ref{infslalom}, in a basic fashion partition $\omega$ into finite, disjoint, non-empty sets $\{J_{n, k} \; | \; k \leq h(n)\}$. Let $J_n = \bigcup_{k \leq h(n)} J_{n, k}$. Let $f':\omega \to \omega^{< \omega}$ be the function defined by $f'(n) = f\upharpoonright J_n$. Then if $\sigma$ is any basic $h$-slalom, let $\sigma '$ be such that on input $n$ gives $h(n)$ many finite partial fuctions $w^n_1,...,w^n_{h(n)}$ with domain $J_n$ where for all $k \leq h(n)$ and $l \in J_n$ $w^n_k(l)$ is the $k^{\rm th}$ greatest number in the set $\sigma (l)$ if such exists and $0$ (say) otherwise. Suppose now towards a contradiction that there is a basic $h$-slalom $\sigma$ such that  $f(n) \in \sigma (n)$ for infinitely many $n$. For each $n$ let $\sigma'(n) = \{w^n_1,...,w^n_{h(n)}\}$. Then define $g_n = \bigcup_{k \leq h(n)} w^n_k \upharpoonright J_{n, k}$ and let $g = \bigcup_{n < \omega} g_n$. Clearly $g$ can built using $\sigma$, the function $h$ and the $J_{n, k}$'s each of which is basic so $g$ is basic. Thus there is a $k$ such that for all $n > k$ we have that $f(n) \neq g(n)$. But, since there are infinitely many $n$ such that $f(n) \in \sigma (n)$, there are infinitely many $n > k$ such that $f(n) \in \sigma (n)$ and therefore it follows that similarly we must have that there are infinitely many $n > k$ such that $f' (n)$ agrees with some $w^n_j$ on some element of their shared domain for some $j \leq h(n)$. But this means $f(k) = g(k)$ for some $k \in J_{n, j}$ for infinitely many $n$'s and $j$'s which is a contradiction.
\end{proof}

\noindent Since this was the final inclusion to prove, Theorem \ref{propdi} is now proved as well.
\end{proof}

Thus, even in this broad context one can construct diagrams for a wide variety of reduction concepts and a correspondence starts to form with the Cicho\'n diagram. This extends the proof given in the case of Turing degrees in \cite{BBNN} and gives a good framework for investigations into various computability reduction concepts. What it does not show, however, is that any of these nodes are non-empty or that the inclusions are strict. Indeed this is not necessarily the case. For instance $\mathcal B_{\leq_T}(\in^*) = \mathcal B_{\leq_T}(\neq^*)$ (see \cite{BBNN}). This is because, by a theorem of Rupprecht, \cite{Rupprecht}, the set $\mathcal B_{\leq_T} (\in^*)$ is simply the high reals, which as I mentioned above is also $\mathcal B (\leq^*)$. The analogue of this fact in the case of the classical Cicho\'n diagram is false since ${\rm add}(\mathcal N)$, the analogue of $\mathcal B_{\leq_T} (\in^*)$, can consistently be less than $\mathfrak{b}$, the analogue of $\mathcal B(\leq^*)$. The authors of \cite{BBNN} take this as evidence that the $\leq_T$-Cicho\'n diagram provides \say{only an analogy, not a full duality} \cite[p. 3]{BBNN} with the classical Cicho\'n diagram. Theorem \ref{propdi} proves the existence of a wide variety of such diagrams, therefore raising the question in each case of how strong the analogy between the reduction diagram and the classical diagram is, and whether we ever get a full duality. This depends on the strength of the reduction since, while the $\leq_T$ diagram gives only an analogy, I show in the next section that in the $\leq_W$ diagram the inclusions proved in Theorem \ref{propdi} constitute the only ones true in every model of ZFC, thereby suggesting something closer to a true duality.

\section{Separations in the $\leq_W$-Cicho\'n Diagram}
From now on fix an inner model $W$. I work in the language of set theory with an extra predicate for $W$ and the theory ZFC($W$), that is ZFC with replacement and comprehension holding for formulas containing $W$. I view $W=L$ as a central case but it turns out that the analysis works out the same for arbitrary $W$. In this section I will be primarily concerned with separations in the diagram, that is forcing to make the inclusions described above proper for the $\leq_W$ reduction concept.

Note that all of the implications discussed in the previous section hold when in the case of the $\leq_W$ reduction concept. For reference, let me state clearly what the unbounding and dominating sets are for the combinatorial relations defined in the last section for $\leq_W$.
\begin{enumerate}
\item
	$\mathcal B(\in^*)$ is the set of reals $x$ such that there is a slalom $\sigma \in W[x]$ that eventually captures all reals in $W$.
\item
	 $\mathcal B(\leq^*)$ is the set of reals $x$ such that there is a real $y \in W[x]$ that eventually dominates all reals in $W$. Such $y$ are called {\em dominating reals} (for $W$).
\item
	$\mathcal B(\neq^*)$ is the set of reals $x$ such that there is a real $y \in W[x]$ that is eventually different from all reals in $W$. Such $y$ are called {\em eventually different reals} (for $W$).
\item
	$\mathcal D(\in^*)$ is the set of reals $x$ such that there is a real $y \in W[x]$ that is not eventually captured by any slalom in $W$.
\item
	$\mathcal D(\leq^*)$ is the set of reals $x$ such that there is a real $y \in W[x]$ that is not eventually dominated by any real in $W$. Such $y$ are called {\em unbounded reals} (for $W$).
\item
	$\mathcal D(\neq^*)$ is the set of reals $x$ such that there is a real $y \in W[x]$ that is equal infinitely often to every real in $W$. Such $y$ are called {\em infinitely-often-equal reals} (for $W$).
	\end{enumerate}

In this section I will study how a variety of known forcing notions over $W$ can create separations in the $\leq_W$-Cicho\'n diagram as described in the previous section. Of course ZFC($W$) cannot prove any separations since if $V=W$ or, more generally $V$ and $W$ have the same reals, every node in the $\leq_W$-diagram will be empty. However, using simple forcings I will show that one can produce a wide variety of possible constellations for the $\leq_W$-diagram. While in most of these cases the results are well known the perspective is sufficiently different that I feel it's worth observing them in this light. In most cases I simply cite the relevant references, however, for Hechler, eventually different forcing and localization forcing the relevant theorem appears to be either new or at least never written down in this way. In particular the section on $\mathbb{LOC}$ contains new results. The main theorem of this section is the following.

\begin{theorem}
The Cicho\'n diagram for $\leq_W$ as described in the previous section is complete for ZFC($W$)-provable implications. In other words if $A$ and $B$ are two nodes in the diagram and there is not an arrow from $A$ to $B$ in the $\leq_W$-diagram then there is a forcing extension of $W$ where $A \nsubseteq B$. Moreover, every single cut, that is two valued split, can be realized by a proper forcing.
\label{mainthm2}
\end{theorem}

Let me note one word on the relation between my diagram and the standard Cicho\'n diagram as commonly studied, for example in \cite{BarJu95}. Here I have focused on the so-called combinatorial nodes as discussed by \cite{BBNN}. As noted in the introduction I view this diagram in correspondence with the classical one via the mapping sending unbounded or dominating families with respect to a certain relation to the sets of reals $x$ such that in $W[x]$ the reals of $W$ are not unbounded or dominating. I have included this fragment of the Cicho\'n diagram to make this analogy clear visually.

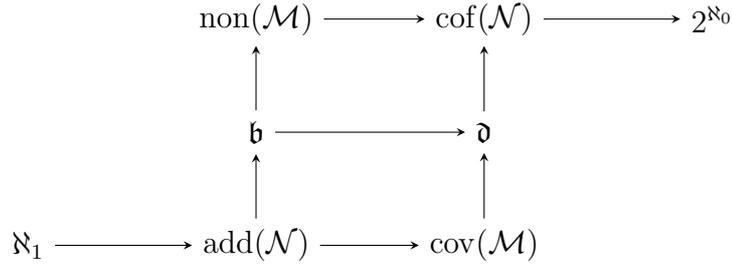
\begin{figure}[h]\label{Figure.Cichon-basic}
\centering
  \begin{tikzpicture}[scale=1.5,xscale=2]
     \draw (0,0) node (empty) {$\aleph_1$}
           (1,0) node (Bin*) {${\rm add}(\mathcal N)$}
           (1,1) node (Bleq*) {$\mathfrak b$}
           (1,2) node (Bneq*) {${\rm non}(\mathcal M)$}
           (2,0) node (Dneq*) {${\rm cov}(\mathcal M)/ \mathfrak{d}(\neq^*)$}
           (2,1) node (Dleq*) {$\mathfrak{d}$}
           (2,2) node (Din*) {${\rm cof}(\mathcal N)$}
           (3,2) node (all) {$2^{\aleph_0}$}
           ;
     \draw[->,>=stealth]
            (empty) edge (Bin*)
            (Bin*) edge (Bleq*)
            (Bleq*) edge (Bneq*)
            (Bleq*) edge (Dleq*)
            (Dneq*) edge (Dleq*)
            (Bin*) edge (Dneq*)
            (Bneq*) edge (Din*)
            (Dleq*) edge (Din*)
            (Din*) edge (all)
            
            ;
  \end{tikzpicture}
  \caption{The Combinatorial Nodes of the Standard Cicho\'n Diagram}
\end{figure}
The details of these correspondences for $\leq_T$ can be found in \cite{BBNN} and similar ideas hold in the present case, with one exception: ${\rm cov}(\mathcal M) / \mathfrak{d}(\neq^*)$. As is well known these cardinals are the same, however Zapletal has shown in \cite{dimtheoryandforcing} that their degree theoretic analogues are in fact different, thus solving a well known problem of Fremlin. In a planned sequel \cite{Switz18b} I will discuss this topic more as well as treat the mising nodes, namely those corresponding to invariants of measure and category. Note however, that the analogy holds between the combinatorial characterizations of the cardinal invariants, not their usual definitions. For example, ${\rm add} (\mathcal N)$ is known to be equal to $\mathfrak{b} (\in^*)$ and it is this cardinal that corresponds to $\mathcal B(\in^*)$. 

\subsection{Sacks Forcing}
The first forcing I will look at is Sacks forcing, $\mathbb S$. Recall that conditions in $\mathbb S$ are perfect trees $T \subseteq 2^{< \omega}$ ordered by inclusion. If $G$ is $\mathbb S$-generic then the unique branch in the intersection of all members of $G$ is called a {\em Sacks real}. I denote such a real $s$. The following theorem is just a reformulation of the Sacks property in terms of the the diagram, see by Lemma 7.3.2 of \cite{BarJu95}.

\begin{theorem}
In the Sacks extension all nodes of $\leq_W$-Cicho\'n diagram other than $\omega^\omega \setminus (\omega^\omega)^W$ are empty.
\begin{figure}[h]\label{Figure.Cichon-basic}
\centering
  \begin{tikzpicture}[scale=1.5,xscale=2]
     \draw (0,0) node (empty) {$\emptyset$}
           (1,0) node (Bin*) {$\mathcal B(\in^*)$}
           (1,1) node (Bleq*) {$\mathcal B(\leq^*)$}
           (1,2) node (Bneq*) {$\mathcal B(\neq^*)$}
           (2,0) node (Dneq*) {$\mathcal D(\neq^*)$}
           (2,1) node (Dleq*) {$\mathcal D(\leq^*)$}
           (2,2) node (Din*) {$\mathcal D(\in^*)$}
           (3,2) node (all) {$\omega^\omega\setminus(\omega^\omega)^W$}
           ;
     \draw[->,>=stealth]
            (empty) edge (Bin*)
            (Bin*) edge (Bleq*)
            (Bleq*) edge (Bneq*)
            (Bin*) edge (Dneq*)
            (Bleq*) edge (Dleq*)
            (Bneq*) edge (Din*)
            (Dneq*) edge (Dleq*)
            (Dleq*) edge (Din*)
            (Din*) edge (all)
            ;
       \draw[thick,dashed,OliveGreen] (2.5,-.5) -- (2.5,2.5);
       \draw[OliveGreen] (-.5,-.5) rectangle (3.5,2.5);
       \draw[draw=none,fill=yellow,fill opacity=.15] (-.5,-.5) rectangle (2.5,2.5);
       \draw[draw=none,fill=red,fill opacity=.15] (2.5,-.5) rectangle (3.5,2.5);
  \end{tikzpicture}
  \caption{After Sacks forcing}
\end{figure}
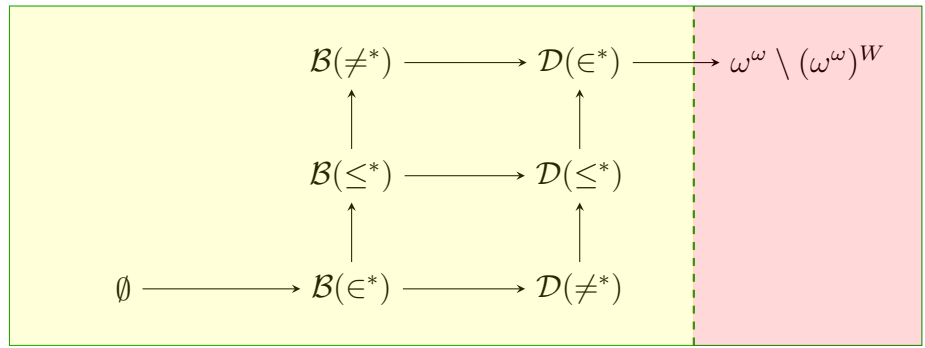
\end{theorem}



\subsection{Cohen Forcing}
Let $\mathbb C = Add(\omega , 1)$ be the forcing to add one Cohen real. The following theorem is essentially standard. The only thing worth pointing out in more detail is that every new real $x \in \omega^\omega \setminus W$ actually codes a Cohen real hence the equality between $\mathcal D (\neq^*)$ and $\omega^\omega \setminus (\omega^\omega)^W$.
\begin{theorem}
Let $c$ be a Cohen real generic over $W$. Then in $W[c]$ the following hold:
\begin{enumerate}
\item   
   $\emptyset = \mathcal B(\in^*) = \mathcal B(\leq^*) = \mathcal B(\neq^*)$ 
\item   
   $\mathcal D(\neq^*) = \mathcal D(\leq^*) = \mathcal D(\in^*) = \{x \; | \; \exists c \in W[x] \; {\rm Cohen \; over \; } W\} = \omega^\omega \setminus (\omega^\omega)^W$
   \end{enumerate}

\end{theorem}

\noindent Thus, the full diagram for Cohen Forcing is:
\begin{figure}[h]\label{Figure.Cichon-basic}
\centering
  \begin{tikzpicture}[scale=1.5,xscale=2]
     \draw (0,0) node (empty) {$\emptyset$}
           (1,0) node (Bin*) {$\mathcal B(\in^*)$}
           (1,1) node (Bleq*) {$\mathcal B(\leq^*)$}
           (1,2) node (Bneq*) {$\mathcal B(\neq^*)$}
           (2,0) node (Dneq*) {$\mathcal D(\neq^*)$}
           (2,1) node (Dleq*) {$\mathcal D(\leq^*)$}
           (2,2) node (Din*) {$\mathcal D(\in^*)$}
           (3,2) node (all) {$\omega^\omega\setminus(\omega^\omega)^W$}
           ;
     \draw[->,>=stealth]
            (empty) edge (Bin*)
            (Bin*) edge (Bleq*)
            (Bleq*) edge (Bneq*)
            (Bin*) edge (Dneq*)
            (Bleq*) edge (Dleq*)
            (Bneq*) edge (Din*)
            (Dneq*) edge (Dleq*)
            (Dleq*) edge (Din*)
            (Din*) edge (all)
            ;
       \draw[thick,dashed,OliveGreen] (1.5,-.5) -- (1.5,2.5);
       \draw[OliveGreen] (-.5,-.5) rectangle (3.5,2.5);
       \draw[draw=none,fill=yellow,fill opacity=.15] (-.5,-.5) rectangle (1.5,2.5);
       \draw[draw=none,fill=Orange,fill opacity=.15] (1.5,-.5) rectangle (3.5,2.5);
  \end{tikzpicture}
  \caption{After Cohen forcing}
\end{figure}
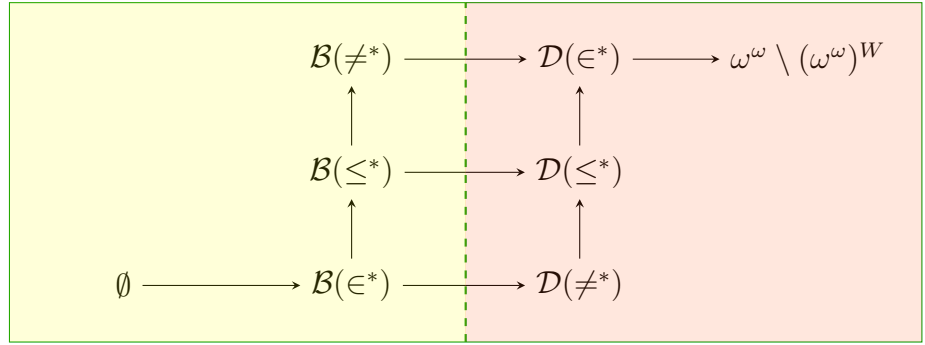


\subsection{Random Real Forcing}
I denote random real forcing by $\mathbb B$. The diagram for random real forcing is as described in the theorem belowand can be proved in a very similar way to that of Cohen forcing using the standard facts found in  \cite[Chapter 3]{BarJu95}. 
\begin{theorem}
Let $r$ be a random real over $W$. Then in $W[r]$ the $\leq_W$-Cicho\'n diagram is determined by the separations $\mathcal B(\in^*) = \mathcal B(\leq^*) = \mathcal D(\neq^*) = \mathcal D(\leq^*) = \emptyset$ and $\mathcal B (\neq^*) = \mathcal D(\in^*) = \omega^\omega \setminus (\omega^\omega)^W$.

\begin{figure}[h]\label{Figure.Cichon-basic}
\centering
  \begin{tikzpicture}[scale=1.5,xscale=2]
     \draw (0,0) node (empty) {$\emptyset$}
           (1,0) node (Bin*) {$\mathcal B(\in^*)$}
           (1,1) node (Bleq*) {$\mathcal B(\leq^*)$}
           (1,2) node (Bneq*) {$\mathcal B(\neq^*)$}
           (2,0) node (Dneq*) {$\mathcal D(\neq^*)$}
           (2,1) node (Dleq*) {$\mathcal D(\leq^*)$}
           (2,2) node (Din*) {$\mathcal D(\in^*)$}
           (3,2) node (all) {$\omega^\omega\setminus(\omega^\omega)^W$}
           ;
     \draw[->,>=stealth]
            (empty) edge (Bin*)
            (Bin*) edge (Bleq*)
            (Bleq*) edge (Bneq*)
            (Bin*) edge (Dneq*)
            (Bleq*) edge (Dleq*)
            (Bneq*) edge (Din*)
            (Dneq*) edge (Dleq*)
            (Dleq*) edge (Din*)
            (Din*) edge (all)
            ;
       \draw[thick,dashed,OliveGreen] (-.5,1.5) -- (3.5,1.5);
       \draw[OliveGreen] (-.5,-.5) rectangle (3.5,2.5);
       \draw[draw=none,fill=yellow,fill opacity=.15] (-.5,-.5) rectangle (3.5,1.5);
       \draw[draw=none,fill=Orange,fill opacity=.15] (-.5,1.5) rectangle (3.5,2.5);
  \end{tikzpicture}
  \caption{After Random Real forcing}
\end{figure}
\label{randomthm}
\end{theorem}



\subsection{Laver Forcing}
Let me now turn to Laver forcing, $\mathbb L$. Recall that conditions in Laver forcing are trees $T\subseteq \omega^{ < \omega}$ with a distinguished {\em stem}, that is, a linearly ordered initial segment, after which there is infinite branching at each node. The order is inclusion. The union of the stems of the trees in a generic for $\mathbb L$ form a real, called a Laver real. Let $l$ denote such a real over $W$. 
\begin{theorem}
Let $l$ be a Laver real over $W$. Then the $\leq_W$ diagram in $W[l]$ has $\emptyset = \mathcal B(\in^*) = \mathcal D(\neq^*)$ and all other nodes are equal to the set of all new reals.
\begin{figure}[h]\label{Figure.Cichon-basic}
\centering
  \begin{tikzpicture}[scale=1.5,xscale=2]
     \draw (0,0) node (empty) {$\emptyset$}
           (1,0) node (Bin*) {$\mathcal B(\in^*)$}
           (1,1) node (Bleq*) {$\mathcal B(\leq^*)$}
           (1,2) node (Bneq*) {$\mathcal B(\neq^*)$}
           (2,0) node (Dneq*) {$\mathcal D(\neq^*)$}
           (2,1) node (Dleq*) {$\mathcal D(\leq^*)$}
           (2,2) node (Din*) {$\mathcal D(\in^*)$}
           (3,2) node (all) {$\omega^\omega\setminus(\omega^\omega)^W$}
           ;
     \draw[->,>=stealth]
            (empty) edge (Bin*)
            (Bin*) edge (Bleq*)
            (Bleq*) edge (Bneq*)
            (Bin*) edge (Dneq*)
            (Bleq*) edge (Dleq*)
            (Bneq*) edge (Din*)
            (Dneq*) edge (Dleq*)
            (Dleq*) edge (Din*)
            (Din*) edge (all)
            ;
       \draw[thick,dashed,OliveGreen] (-.5,.5) -- (3.5,.5);
       \draw[OliveGreen] (-.5,-.5) rectangle (3.5,2.5);
       \draw[draw=none,fill=yellow,fill opacity=.15] (-.5,-.5) rectangle (3.5,.5);
       \draw[draw=none,fill=Orange,fill opacity=.15] (-.5,.5) rectangle (3.5,2.5);
  \end{tikzpicture}
  \caption{After Laver forcing}
\end{figure}
\label{laver}
\end{theorem}

As before this theorem follows from well known facts about $\mathbb L$. In particular the Laver property, \cite[Definition 6.3.27]{BarJu95}, which holds of $\mathbb L$ \cite[Theorem 7.3.29]{BarJu95}, implies that there are no infinitely often equal reals in $W[l]$. Thus it suffices to note that $l$ is dominating and, by \cite[Theorem 7]{Gro87}, that Laver reals satisfy the following minimality property: if $x$ is a real such that $x \in W[l] \setminus W$ then $l \in W[x]$. Therefore every new real constructs a dominating real, hence the equality between $\mathcal B(\leq^*)$ and $\omega^\omega \setminus (\omega^\omega)^W$.

\subsection{Rational Perfect Tree Forcing}
Next I look at is Miller's rational perfect tree forcing, $\mathbb{PT}$. Recall that $\mathbb{PT}$ is the set of perfect trees $T \subseteq \omega^{< \omega}$ so that for all $s \in T$ there is a $t \supseteq s$ with $\omega$-many immediate successors. The order is inclusion and the unique branch through the trees in the generic is called a {\em Miller real}. Let us denote such a real by $m$.
\begin{theorem}
Let $m$ be a Miller real over $W$. Then the $\leq_W$ diagram in $W[m]$ is determined by $\emptyset = \mathcal B(\neq^*) = \mathcal D(\neq^*)$ and all other nodes are equal to the set of all new reals.
\begin{figure}[h]\label{Figure.Cichon-basic}
\centering
  \begin{tikzpicture}[scale=1.5,xscale=2]
     \draw (0,0) node (empty) {$\emptyset$}
           (1,0) node (Bin*) {$\mathcal B(\in^*)$}
           (1,1) node (Bleq*) {$\mathcal B(\leq^*)$}
           (1,2) node (Bneq*) {$\mathcal B(\neq^*)$}
           (2,0) node (Dneq*) {$\mathcal D(\neq^*)$}
           (2,1) node (Dleq*) {$\mathcal D(\leq^*)$}
           (2,2) node (Din*) {$\mathcal D(\in^*)$}
           (3,2) node (all) {$\omega^\omega\setminus(\omega^\omega)^W$}
           ;
     \draw[->,>=stealth]
            (empty) edge (Bin*)
            (Bin*) edge (Bleq*)
            (Bleq*) edge (Bneq*)
            (Bin*) edge (Dneq*)
            (Bleq*) edge (Dleq*)
            (Bneq*) edge (Din*)
            (Dneq*) edge (Dleq*)
            (Dleq*) edge (Din*)
            (Din*) edge (all)
            ;
       \draw[thick,dashed,OliveGreen] (1.5,.5) -- (3.5,.5);
       \draw[thick,dashed,OliveGreen] (1.5,.5) -- (1.5,2.5);
       \draw[OliveGreen] (-.5,-.5) rectangle (3.5,2.5);
       \draw[draw=none,fill=yellow,fill opacity=.15] (-.5,-.5) rectangle (3.5,.5);
       \draw[draw=none,fill=yellow,fill opacity=.15] (-.5,.5) rectangle (1.5,2.5);
       \draw[draw=none,fill=Orange,fill opacity=.15] (1.5,.5) rectangle (3.5,2.5);
  \end{tikzpicture}
  \caption{After rational perfect forcing}
\end{figure}
\label{miller}
\end{theorem}

This is proved in the same way as for Laver forcing. It suffices to note that $\mathbb{PT}$ adds no eventually different real, see \cite[Theorem 7.3.46, Part 1]{BarJu95}, $\mathbb{PT}$ adds no infinitely often equal real as it enjoys the Laver property (\cite[Theorem 7.3.45]{BarJu95}) and $m$ is of minimal degree, see \cite[Theorem 3]{Gro87}.


\subsection{Hechler Forcing}
Let $\mathbb D$ be Hechler forcing and let $d$ be the associated dominating real. Recall that conditions of $\mathbb D$ are pairs $(p, \mathcal F)$ where $p$ is a finite partial function from $\omega$ to $\omega$ and $\mathcal F$ is a finite family of elements of $\omega^\omega$. The order is given by $(q, \mathcal G) \leq_\mathbb D (p, \mathcal F)$ if and only if $q \supseteq p$, $\mathcal G \supseteq \mathcal F$ and for all $n \in {\rm dom}(q) \setminus {\rm dom}(p)$ and all $f \in \mathcal F$, $q(n) > f(n)$. Note that since $d$ is dominating, $d \in \mathcal B(\leq^*)$. 
\begin{theorem}
After Hechler forcing over $W$ the $\leq_W$-diagram has
\begin{enumerate}
\item
 $\emptyset = \mathcal B(\in^*)$, 
 \item
 $\mathcal B(\leq^*) = \mathcal B(\neq^*)$ and
 \item
 $\mathcal D(\neq^*) = \mathcal D(\leq^*) = \mathcal D(\in^*) = \omega^\omega \setminus (\omega^\omega)^W$. 
 \end{enumerate}

\begin{figure}[h]\label{Figure.Cichon-basic}
\centering
  \begin{tikzpicture}[scale=1.5,xscale=2]
     \draw (0,0) node (empty) {$\emptyset$}
           (1,0) node (Bin*) {$\mathcal B(\in^*)$}
           (1,1) node (Bleq*) {$\mathcal B(\leq^*)$}
           (1,2) node (Bneq*) {$\mathcal B(\neq^*)$}
           (2,0) node (Dneq*) {$\mathcal D(\neq^*)$}
           (2,1) node (Dleq*) {$\mathcal D(\leq^*)$}
           (2,2) node (Din*) {$\mathcal D(\in^*)$}
           (3,2) node (all) {$\omega^\omega\setminus(\omega^\omega)^W$}
           ;
     \draw[->,>=stealth]
            (empty) edge (Bin*)
            (Bin*) edge (Bleq*)
            (Bleq*) edge (Bneq*)
            (Bin*) edge (Dneq*)
            (Bleq*) edge (Dleq*)
            (Bneq*) edge (Din*)
            (Dneq*) edge (Dleq*)
            (Dleq*) edge (Din*)
            (Din*) edge (all)
            ;
       \draw[thick,dashed,OliveGreen] (1.5,-.5) -- (1.5,2.5);
       \draw[thick,dashed,OliveGreen] (-.5,.5) -- (1.5,.5);
       \draw[OliveGreen] (-.5,-.5) rectangle (3.5,2.5);
       \draw[draw=none,fill=yellow,fill opacity=.15] (-.5,-.5) rectangle (1.5,.5);
       \draw[draw=none,fill=blue,fill opacity=.15] (-.5,.5) rectangle (1.5,2.5);
       \draw[draw=none,fill=Orange,fill opacity=.15] (1.5,-.5) rectangle (3.5,2.5);
  \end{tikzpicture}
  \caption{After Hechler forcing}
\end{figure}
\label{hechler}
\end{theorem}

\begin{proof}
There are three things to observe: first that Hechler forcing does not add an eventually different real that is not dominating, second that it does not add a slalom eventually capturing all ground model reals and the third is that every subforcing of Hechler forcing adds an infinitely often equal real. The first of these is Corollary 13 of \cite{BL11}. The second is well known, see for example Theorem 3.6 of \cite{ShJdKM}. The third follows from \cite[Theorem 8.1]{Pal13} which states that every subforcing of Hechler forcing adds a Cohen real.

\end{proof}

\subsection{Eventually Different Forcing}
Let $\mathbb E$ be {\em eventually different forcing}, which is defined like $\mathbb D$ except that stems of extensions need simply be eventually different from the reals in the second component, not dominating. I will show that:

\begin{theorem}
Assume that every set of reals in $L(\mathbb R)$ has the Baire property (this is implied by sufficiently large cardinals). Let $e$ be an $\mathbb E$-generic real over $W$. Then in $W[e]$ the following hold: 
\begin{enumerate}
\item
$\mathcal B(\in^*) = \mathcal B(\leq^*) = \emptyset$, 
\item
$\mathcal B(\neq^*) \subsetneq \mathcal D(\neq^*) = \mathcal D(\leq^*) = \mathcal D(\in^*) = \omega^\omega \setminus (\omega^\omega)^W$.
\end{enumerate}
Thus in particular the full diagram for eventually different forcing is as shown in Figure 7.

\begin{figure}[h]\label{Figure.Cichon-basic}
\centering
  \begin{tikzpicture}[scale=1.5,xscale=2]
     \draw (0,0) node (empty) {$\emptyset$}
           (1,0) node (Bin*) {$\mathcal B(\in^*)$}
           (1,1) node (Bleq*) {$\mathcal B(\leq^*)$}
           (1,2) node (Bneq*) {$\mathcal B(\neq^*)$}
           (2,0) node (Dneq*) {$\mathcal D(\neq^*)$}
           (2,1) node (Dleq*) {$\mathcal D(\leq^*)$}
           (2,2) node (Din*) {$\mathcal D(\in^*)$}
           (3,2) node (all) {$\omega^\omega\setminus(\omega^\omega)^W$}
           ;
     \draw[->,>=stealth]
            (empty) edge (Bin*)
            (Bin*) edge (Bleq*)
            (Bleq*) edge (Bneq*)
            (Bin*) edge (Dneq*)
            (Bleq*) edge (Dleq*)
            (Bneq*) edge (Din*)
            (Dneq*) edge (Dleq*)
            (Dleq*) edge (Din*)
            (Din*) edge (all)
            ;
       \draw[thick,dashed,OliveGreen] (1.5,-.5) -- (1.5,2.5);
       \draw[thick,dashed,OliveGreen] (-.5,1.5) -- (1.5,1.5);
       \draw[OliveGreen] (-.5,-.5) rectangle (3.5,2.5);
       \draw[draw=none,fill=yellow,fill opacity=.15] (-.5,-.5) rectangle (1.5,1.5);
       \draw[draw=none,fill=blue,fill opacity=.15] (-.5,1.5) rectangle (1.5,2.5);
       \draw[draw=none,fill=Orange,fill opacity=.15] (1.5,-.5) rectangle (3.5,2.5);
  \end{tikzpicture}
  \caption{After Eventually Different forcing}
\end{figure}
\label{eforcing}
\end{theorem}

This is proved in a way exactly analogous to Hechler forcing noting that $\mathbb{E}$ does not add a dominating real (see \cite[p. 385]{BarJu95}) and, assuming the large cardinal hypothesis, Palumbo's result about $\mathbb D$ stating that every subforcing of $\mathbb{D}$ adds a Cohen real can be extended to $\mathbb E$.

\begin{lemma}
Assume that every set of reals in $L(\mathbb R)$ has the property of Baire. Then in every nontrivial intermediate model between $W$ and $W[e]$ there is a real $c$ which is $\mathbb C$-generic over $W$.
\end{lemma}

A proof of this is sketched in \cite[pg 38]{Pal13} for $\mathbb D$ but the reader will notice that it goes through equally well for $\mathbb E$. Indeed the centerpiece of the argument involves a fact, due to Shelah and Gitik \cite[Proposition 4.3]{GS93} that given any sufficiently well-defined $\sigma$-centered forcing $\mathbb P$, if certain filters of $\mathbb P$ in $L(\mathbb R)$ have the property of Baire, then $\mathbb P$ will add a Cohen real. It is not hard to see from the combination of the Gitik-Shelah and the Palumbo arguments that \say{sufficiently well defined} includes all subforcings of $\mathbb E$. Thus, assuming all sets of reals have the property of Baire the result goes through.

Using this lemma, by the same argument given for $\mathbb D$, we have the proof of Theorem \ref{eforcing}.

The use of large cardinals here is unfortunate and I hope it can be improved on. The result for $\mathbb D$ (that avoids large cardinals) uses the tree version of Hechler forcing and I do not know of an analogous one for $\mathbb E$. Let me note however that even without large cardinals I have shown that there is a model realizing the cut determined by $\mathcal B(\in^*) = \mathcal B(\neq^*) = \emptyset$.

\subsection{Localization Forcing}
In this section I study {\em Localization forcing}, the forcing to add a generic slalom capturing all ground model reals. This section is the only forcing containing essentially new results to the best of my knowledge.
\begin{definition}[Localization Forcing (cf \cite{BL11})]
The localization forcing $\mathbb{LOC}$ is defined as the set of pairs $(s, \mathcal F)$ such that $s \in ([\omega]^{<\omega})^{< \omega}$ is a finite sequence with $|s(n)| \leq n$ for all $n < |s|$ and $\mathcal F$ is a a finite family of functions in Baire space with $|\mathcal F| \leq |s|$. The order is $(t, \mathcal G) \leq_{\mathbb{LOC}} (s, \mathcal F)$ if and only if $t \supseteq s$, $\mathcal G \supseteq \mathcal F$ and $f(n) \in t(n)$ for all $f \in \mathcal F$ and all $n \in |t| \setminus |s|$. We think of the first component as a finite approximation to a slalom we are trying to build and as such I will often refer to the length of the sequence as its \say{domain} and write ${\rm dom}(s)$.
\end{definition}

Unfortunately I do not have a full characterization of the diagram in the case of $\mathbb{LOC}$. The following theorem summarizes the state of knowledge.
\begin{theorem}
Let $\sigma$ be a slalom which is $\mathbb{LOC}$-generic over $W$. Then in $W[\sigma]$ all the nodes in the diagram are non-empty and we have that $\mathcal B(\in^*)$ is a proper subset of $\mathcal B(\leq^*)$ and $\mathcal D(\neq^*)$. Also $\mathcal B(\leq^*) \subsetneq \mathcal B(\neq^*)$ and $\mathcal D (\leq^*) \subsetneq \mathcal D(\in^*)$. In particular, Figure 12 is a partial diagram for $\mathbb{LOC}$. 
\begin{figure}[h]\label{Figure.Cichon-basic}
\centering
  \begin{tikzpicture}[scale=1.5,xscale=2]
     \draw (0,0) node (empty) {$\emptyset$}
           (1,0) node (Bin*) {$\mathcal B(\in^*)$}
           (1,1) node (Bleq*) {$\mathcal B(\leq^*)$}
           (1,2) node (Bneq*) {$\mathcal B(\neq^*)$}
           (2,0) node (Dneq*) {$\mathcal D(\neq^*)$}
           (2,1) node (Dleq*) {$\mathcal D(\leq^*)$}
           (2,2) node (Din*) {$\mathcal D(\in^*)$}
           (3,2) node (all) {$\omega^\omega\setminus(\omega^\omega)^W$}
           (2.5, .5) node[fill=white] (?) {$?$}
           (2.5, 1.5) node[fill=white] (?') {$?$}
           ;
     \draw[->,>=stealth]
            (empty) edge (Bin*)
            (Bin*) edge (Bleq*)
            (Bleq*) edge (Bneq*)
            (Bin*) edge (Dneq*)
            (Bleq*) edge (Dleq*)
            (Bneq*) edge (Din*)
            (Dneq*) edge (Dleq*)
            (Dleq*) edge (Din*)
            (Din*) edge (all)
            ;
        \draw[thick,dashed,OliveGreen] (-.5,.5) -- (2.45,.5);
        \draw[thick,dashed,OliveGreen] (2.55,.5) -- (3.5,.5);
       \draw[thick,dashed,OliveGreen] (1.5,-.5) -- (1.5,2.5);
       \draw[thick,dashed,OliveGreen] (-.5,1.5) -- (.55,1.5);
       \draw[thick,dashed,OliveGreen] (.55,1.5) -- (2.43,1.5);
       \draw[thick,dashed,OliveGreen] (2.55,1.5) -- (3.5,1.5);
       \draw[thick,dashed,OliveGreen] (.5,-.5) -- (.5,.5);
       \draw[thick,dashed,OliveGreen] (2.5,1.65) -- (2.5,2.5);
       \draw[OliveGreen] (-.5,-.5) rectangle (3.5,2.5);
       \draw[draw=none,fill=yellow,fill opacity=.15] (-.5,-.5) rectangle (.5,.5);
       \draw[draw=none,fill=green,fill opacity=.15] (.5,-.5) rectangle (1.5,.5);
       \draw[draw=none,fill=blue,fill opacity=.15] (-.5,.5) rectangle (1.5,1.5);
        \draw[draw=none,fill=red,fill opacity=.15] (-.5,1.5) rectangle (1.5,2.5);
       \draw[draw=none,fill=orange,fill opacity=.15] (1.5, 1.5) rectangle (2.5,2.5);
  \end{tikzpicture}
  \caption{Partial diagram after Localization forcing}
  \label{locpic}
\end{figure}
\label{locthm}
\end{theorem}

Proving this theorem amounts to showing that $\mathbb{LOC}$ adds $\mathbb B$, $\mathbb D$ and $\mathbb E$ generics. I start with $\mathbb D$. Notice first that $\mathbb{LOC}$ adds a dominating real. Indeed if $\sigma$ is a generic slalom in $W^{\mathbb{LOC}}$ then $d(n) :={\rm max} \; \sigma (n)$ has this property. This is actually a Hechler real:
\begin{lemma}
Let $\sigma \in W^\mathbb{LOC}$ be a generic slalom eventually capturing all ground model reals. Then, $d(n):={\rm max}\; \sigma (n)$ is $\mathbb D$-generic over $W$.
\end{lemma}

To prove this I will need a simplified version of $\mathbb D$: in the first component of a condition I will assume that the domain is a finite initial segment of $\omega$ and instead of having the second component of a condition of $\mathbb D$ be a finite family of functions, it will be a single function. Then $(q, g) \leq_\mathbb D (p, f)$ if and only if $q$ extends $p$, for all $n \in {\rm dom}(q) \setminus {\rm dom}(p)$, $q(n) \geq f(n)$ and for all $n \in \omega$, and $g(n) \geq f(n)$. It's not hard to see that this version of $\mathbb D$ is forcing equivalent to the original one I defined.

\begin{proof}
Recall that a {\em projection} $\pi:\mathbb P \to \mathbb Q$ between two posets is an order preserving map which sends the maximal element of $\mathbb P$ to the maximal element of $\mathbb Q$ and for all $p \in \mathbb P$ and all $q \leq \pi(p)$ there is some $\overline{p} \leq p$ such that $\pi(\overline{p}) \leq q$. If a projection exists between $\mathbb P$ and $\mathbb Q$ then the image $\pi '' G$ of a $\mathbb P$-generic filter generates a $\mathbb Q$-generic filter. Therefore to prove the lemma it suffices to show that the map $\pi:\mathbb{LOC} \to \mathbb D$ such that $\pi(s, \mathcal F) = (n \mapsto {\rm max} \; s(n), \Sigma \mathcal F)$ where $\Sigma \mathcal F$ is the pointwise sum, is a projection. To see why, note that if $(s, \mathcal F) \in \mathbb{LOC}$ and let, for all $n \in {\rm dom} (s)$, $p(n) = {\rm max} \; s(n)$ and let $f = \Sigma \mathcal F$. Since $\mathcal F$ is finite this is well defined. Then the pair $(p, f)$ is a $\mathbb D$ condition and the union of all conditions such defined from elements of the $\mathbb{LOC}$ generic defining $\sigma$ is the $d$ from the statement of the lemma. 

It is routine to check that $\pi(1_{\mathbb{LOC}}) =  1_\mathbb D$ and that the map $\pi$ is order preserving. The difficulty is in verifying the third condition of projections. To this end, let $(s, \mathcal F) \in \mathbb{LOC}$ and let $(p, f)=\pi(s, \mathcal F)$. Let $(p', f') \leq (p, f)$ and let $D \subseteq \mathbb D$ be a set of conditions which is dense below $(p', f')$. It suffices to find a strengthening $(t, \mathcal G)$ of $(s, \mathcal F)$, such that $(n \mapsto {\rm max} \; t(n), \Sigma G) \in D$. To do this, first, find a function $g:\omega \to \omega$ such that for all $n \notin {\rm dom}(p)$, $g(n) > n + f$ and otherwise is at least as big as $f$. Then, $(p, g)$ strengthens $(p, f)$ and is compatible with $(p', f')$. Let $(q, h) \in D$ strengthen $(p, g)$. 

Now, we can build our new $\mathbb{LOC}$ condition. Define $H:\omega \to \omega$ by $H(n) = h(n) - f(n)$. Notice that since $g(n)$ was assumed to be bigger than $f(n)$ for all $n$ and $h(n) \geq g(n)$ since it is a strengthening it follows that $H$ is in fact always positive. Moreover, $f + H = \Sigma \mathcal F + H = h$. It remains to show that there is a $t \supseteq s$ such that ${\rm dom}(t) = {\rm dom}(q)$, for all $n \in {\rm dom}(t)$, ${\rm max} \; t(n) = q(n)$ and for all $n \in {\rm dom}(t) \setminus {\rm dom}(s)$ and all $f\in \mathcal F$, $f(n) \in t(n)$. Once this has been done $(t, \mathcal F \cup \{H\})$ will be the desired condition. I claim that this is all possible. I will describe a $t$ extending $s$ be defined on the domain of $q$ (by construction, the domain of $q$ contains that of $s$). Without loss of generality $|{\rm dom}(q)| > |{\rm dom}(s)| + 2$. Thus, the domain of $t$ will be large enough to accomodate the side condition $\mathcal F \cup \{H\}$. Let $|\mathcal F| = k$ and enumerate $\mathcal F = \{f_0,...,f_{k-1}\}$. Note that $k < n$ for all $n \in {\rm dom}(q) \setminus {\rm dom}(s)$. Now, for each $n \in {\rm dom}(q) \setminus {\rm dom}(s)$, let me define $t(n)$. Notice first that one must put in all $k$ numbers $\{f_0(n),...,f_{k-1}(n)\}$ and we also want ${\rm max} \; t(n) = q(n)$ so add this in too. Since $n > k$, one needs to simply add $n -k-1$ additional numbers $\{j_0,...,j_{n-k-2}\}$ such that each one is less than $q(n)$ and different from all numbers in the set $\{f_0(n),...,f_{k-1}(n), q(n)\}$. This is possible however, since by construction $q(n) \geq g(n)$ for all $n \notin {\rm dom}(p)$ and $g(n) > n + \Sigma_{i < k} f_i(n)$ on this domain. Thus, there must be at least $n$ between the maximum of the $f_i(n)$'s and $q(n)$, which is more than we needed.
\end{proof}

Now, I show that $\mathbb{LOC}$ adds an $\mathbb E$-generic real. This fact was first told to me (without proof) in private communication with J. Brendle. I thank him for pointing it out to me.

\begin{lemma}
The forcing $\mathbb{LOC}$ adds an $\mathbb E$-generic real.
\end{lemma}

\begin{proof}
Given a condition $(s, \mathcal F) \in \mathbb{LOC}$ define a stem for an $\mathbb E$-condition as $p_s : {\rm dom} (s) \to \omega$ by letting for all $n \in {\rm dom}(s)$ $p_s(n)$ be equal to the $k^{\rm th}$ natural number $m$ not in the set $s(n)$ where the pointwise sum $\Sigma s(n) \equiv k \; {\rm mod} \; n$. We claim that the map $\pi:\mathbb{LOC} \to \mathbb{E}$ defined by $\pi(s, \mathcal F) = (p_s, \mathcal F)$ is a projection. Clearly the maximal condition is sent to the maximal condition and this map is order preserving. Let $(s, \mathcal F) \in \mathbb{LOC}$, and let $(q, \mathcal G) \leq_\mathbb E (p_s, \mathcal F)$. We need to show that there is a strengthening of $(q, \mathcal G)$ in the image of $\pi$. To this end, note that we can assume with out loss that $|\mathcal G| < {\rm dom}(q)$ since otherwise we can strengthen to make this true. Now, define a partial slalom as follows: $s_q:{\rm dom}(q) \to [\omega]^{<\omega}$. For $n \in {\rm dom}(p)$ let $s_q(n) = s(n)$. For $n \notin {\rm dom}(p)$ let $q(n) = m$ and suppose that $m$ is the $k^{\rm th}$ not in $\{f(n) \; |\; f \in \mathcal F\}$ and suppose that this set has size $l < n$ (the $<$ follows from the fact that $(p, \mathcal F)$ is in the image of $\pi$). Then, pick $n-l$ numbers $m_{l}, m_{l+1},...,m_{n-1}$ all greater than every $f(n)$ for $f \in \mathcal F$ and not equal to $m$ so that $\Sigma_{f \in \mathcal f} f(n) + \Sigma_{i=l}^{n-1} m_i \equiv k \; {\rm mod}\; n$. This can be accomplished, for instance, as follows: if $\Sigma_{f \in \mathcal F} f(n) \equiv j \; {\rm mod} \; n$ then let $m_l \equiv k - j \; {\rm mod} \; n$ greater than all the $f(n)$'s and let all other $m_i$'s be multiples of $n$. Finally let $s_q(n) = \{f(n) \; | \; f \in \mathcal F\} \cup \{m_{l},...,m_{n-1}\}$. Then $(s_q, \mathcal G) \leq (s, \mathcal F)$ and $\pi(s_q, \mathcal G) = (q, \mathcal G)$ as needed.
\end{proof}

Finally,
\begin{lemma}
Any forcing adding a slalom eventually capturing all ground model reals adds a random real. In particular $\mathbb{LOC}$ adds a random real.
\end{lemma}

\begin{proof}
By Corollary 3.2 of \cite{ShJdKM} adding a slalom eventually capturing all ground model reals is equivalent to adding a Borel null set which covers all Borel null sets coded in the ground model. Let $N \subseteq \omega^\omega$ be such a null set and let $y \notin N$. Then $y$ is not in any ground model null set so $y$ is a random real. 
\end{proof}

Combining all of these results then proves Theorem \ref{locthm} since both $\mathbb D$ and $\mathbb E$ add Cohen reals realizing the split down the middle in Figure \ref{locpic} and $\mathbb B$ adds a bounded real not caught in any old slalom so $\mathcal D(\leq^*)$ is strictly contained in $\mathcal D(\in^*)$.

As an aside notice that there seem to be other eventually different reals added by $\mathbb{LOC}$:
\begin{observation}
Let $\sigma \in W^\mathbb{LOC}$ be a generic slalom eventually capturing all ground model reals. Let $a(n)$ be defined as the least $k \notin \sigma (n)$. Then $a$ is a real which is eventually different from all ground model reals but is not an $\mathbb E$-generic real.
\label{evdiffloc}
\end{observation}

\begin{proof}
First notice that the $a$ described in the theorem is in fact eventually different from all ground model reals since every real eventually is captured by $\sigma$ and after that point $a$ is different from it. Moreover, notice that $a$ is not only not dominating over the ground model reals but actually not even unbounded since, given any real $f \in W$ growing faster than the identity ($n \mapsto n+2$ even), the least $k$ not in $\sigma (n)$ must be less than $f(n)$ since $|\sigma(n)| = n$. From this it follows that $a$ is not an $\mathbb E$-generic since it is not unbounded. 
\end{proof}

This lemma is somewhat surprising and indeed I do not know exactly what the forcing adding the real $a$ is or if it is a previously studied notion. In particular, I don't know if this real is random over $W$, though I conjecture that it is.

\subsection{Cuts in the Diagram and the Analogy with Cardinal Characteristics}
Let me finish this section by noting that it follows from what I have shown that the ZFC($W$)-provable subset implications implied by Theorem \ref{propdi} are the only ones. In other words, Theorem \ref{mainthm2} is proved. Indeed a simply inspection of the diagrams above show that every implication shown in Figure 1 is consistently strict and no other implications are true in every $V$ extending $W$. This shows also that the analogue discussed in the previous section holds in a robust way with the traditional Cicho\'n diagram. In fact, we can actually show that a stronger fact is true.

\begin{theorem}
All cuts consistent with the diagram are consistent with ZFC($W$) in the following sense: Given any collection $N$ of (not $\emptyset$)-nodes in the diagram which are closed upwards under $\subseteq$ there is a proper forcing $\mathbb P$ in $W$ so that forcing with $\mathbb P$ over $W$ results in all and only the nodes in $N$ being nonempty. See Figure \ref{allcuts} for a pictoral representation.
\end{theorem}

Note that this is slightly weaker than the sense of cuts I have been considering above since I'm making no distinction between various non-empty nodes after forcing. Also, as with several of the results in this section, this is essentially folklore, but we include it to cement the perspective.
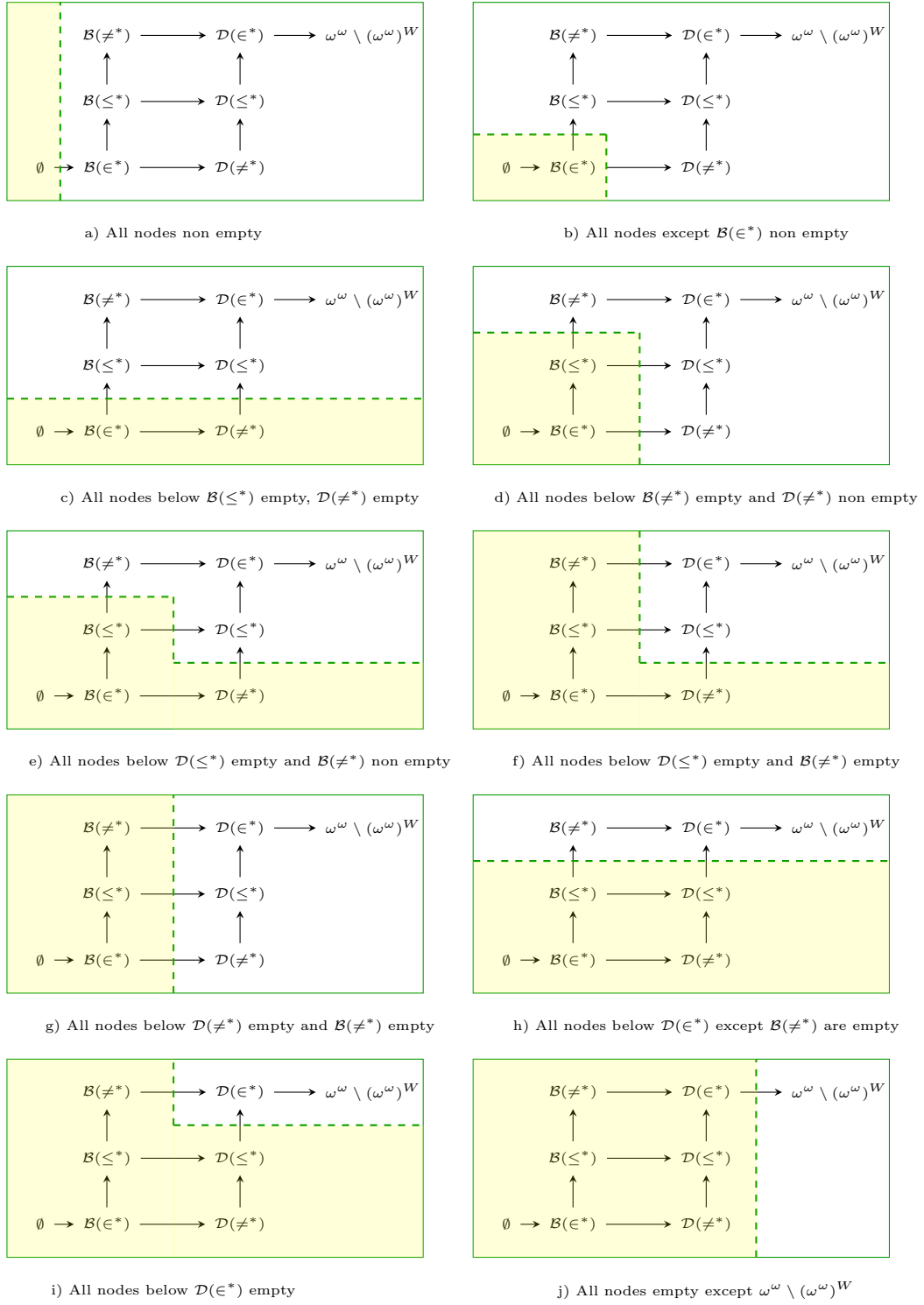
\begin{figure}\label{Figure.Cichon-basic}
\centering
  \begin{tikzpicture}
  \begin{scope}[shift={(0, 16)}]
  
     \draw (0,0) node (empty) {\tiny{$\emptyset$}}
           (1,0) node (Bin*) {\tiny{$\mathcal B(\in^*)$}}
           (1,1) node (Bleq*) {\tiny{$\mathcal B(\leq^*)$}}
           (1,2) node (Bneq*) {\tiny{$\mathcal B(\neq^*)$}}
           (3,0) node (Dneq*) {\tiny{$\mathcal D(\neq^*)$}}
           (3,1) node (Dleq*) {\tiny{$\mathcal D(\leq^*)$}}
           (3,2) node (Din*) {\tiny{$\mathcal D(\in^*)$}}
           (5,2) node (all) {\tiny{$\omega^\omega\setminus(\omega^\omega)^W$}}
           (2, -1) node (label) {\tiny{a) All nodes non empty}}
           ;
     \draw[->,>=stealth]
            (empty) edge (Bin*)
            (Bin*) edge (Bleq*)
            (Bleq*) edge (Bneq*)
            (Bin*) edge (Dneq*)
            (Bleq*) edge (Dleq*)
            (Bneq*) edge (Din*)
            (Dneq*) edge (Dleq*)
            (Dleq*) edge (Din*)
            (Din*) edge (all)
            ;
       \draw[thick,dashed,OliveGreen] (.3,-.5) -- (.3,2.5);
       \draw[OliveGreen] (-.5,-.5) rectangle (5.75,2.5);
       \draw[draw=none,fill=yellow,fill opacity=.15] (-.5,-.5) rectangle (.3,2.5);

       \end{scope}
       
        \begin{scope}[shift={(7, 0)}]
  
     \draw (0,0) node (empty) {\tiny{$\emptyset$}}
           (1,0) node (Bin*) {\tiny{$\mathcal B(\in^*)$}}
           (1,1) node (Bleq*) {\tiny{$\mathcal B(\leq^*)$}}
           (1,2) node (Bneq*) {\tiny{$\mathcal B(\neq^*)$}}
           (3,0) node (Dneq*) {\tiny{$\mathcal D(\neq^*)$}}
           (3,1) node (Dleq*) {\tiny{$\mathcal D(\leq^*)$}}
           (3,2) node (Din*) {\tiny{$\mathcal D(\in^*)$}}
           (5,2) node (all) {\tiny{$\omega^\omega\setminus(\omega^\omega)^W$}}
           (3, -1) node (label) {\tiny{j) All nodes empty except $\omega^\omega \setminus (\omega^\omega)^W$}}
           ;
     \draw[->,>=stealth]
            (empty) edge (Bin*)
            (Bin*) edge (Bleq*)
            (Bleq*) edge (Bneq*)
            (Bin*) edge (Dneq*)
            (Bleq*) edge (Dleq*)
            (Bneq*) edge (Din*)
            (Dneq*) edge (Dleq*)
            (Dleq*) edge (Din*)
            (Din*) edge (all)
            ;
       \draw[thick,dashed,OliveGreen] (3.75,-.5) -- (3.75,2.5);
       \draw[OliveGreen] (-.5,-.5) rectangle (5.75,2.5);
       \draw[draw=none,fill=yellow,fill opacity=.15] (-.5,-.5) rectangle (3.75,2.5);

       \end{scope}
       
       \begin{scope}[shift={(0, 0)}]
  
     \draw (0,0) node (empty) {\tiny{$\emptyset$}}
           (1,0) node (Bin*) {\tiny{$\mathcal B(\in^*)$}}
           (1,1) node (Bleq*) {\tiny{$\mathcal B(\leq^*)$}}
           (1,2) node (Bneq*) {\tiny{$\mathcal B(\neq^*)$}}
           (3,0) node (Dneq*) {\tiny{$\mathcal D(\neq^*)$}}
           (3,1) node (Dleq*) {\tiny{$\mathcal D(\leq^*)$}}
           (3,2) node (Din*) {\tiny{$\mathcal D(\in^*)$}}
           (5,2) node (all) {\tiny{$\omega^\omega\setminus(\omega^\omega)^W$}}
           (2, -1) node (label) {\tiny{i) All nodes below $\mathcal D(\in^*)$ empty}}
           ;
     \draw[->,>=stealth]
            (empty) edge (Bin*)
            (Bin*) edge (Bleq*)
            (Bleq*) edge (Bneq*)
            (Bin*) edge (Dneq*)
            (Bleq*) edge (Dleq*)
            (Bneq*) edge (Din*)
            (Dneq*) edge (Dleq*)
            (Dleq*) edge (Din*)
            (Din*) edge (all)
            ;
       \draw[thick,dashed,OliveGreen] (2., 1.5) -- (2,2.5);
       \draw[thick,dashed,OliveGreen] (2., 1.5) -- (5.75,1.5);
       \draw[OliveGreen] (-.5,-.5) rectangle (5.75,2.5);
       \draw[draw=none,fill=yellow,fill opacity=.15] (-.5,-.5) rectangle (2,2.5);
       \draw[draw=none,fill=yellow,fill opacity=.15] (2,-.5) rectangle (5.75,1.5);

       \end{scope}
       
       \begin{scope}[shift={(7, 4)}]
  
     \draw (0,0) node (empty) {\tiny{$\emptyset$}}
           (1,0) node (Bin*) {\tiny{$\mathcal B(\in^*)$}}
           (1,1) node (Bleq*) {\tiny{$\mathcal B(\leq^*)$}}
           (1,2) node (Bneq*) {\tiny{$\mathcal B(\neq^*)$}}
           (3,0) node (Dneq*) {\tiny{$\mathcal D(\neq^*)$}}
           (3,1) node (Dleq*) {\tiny{$\mathcal D(\leq^*)$}}
           (3,2) node (Din*) {\tiny{$\mathcal D(\in^*)$}}
           (5,2) node (all) {\tiny{$\omega^\omega\setminus(\omega^\omega)^W$}}
           (3, -1) node (label) {\tiny{h) All nodes below $\mathcal D(\in^*)$ except $\mathcal B(\neq^*)$ are empty}}
           ;
     \draw[->,>=stealth]
            (empty) edge (Bin*)
            (Bin*) edge (Bleq*)
            (Bleq*) edge (Bneq*)
            (Bin*) edge (Dneq*)
            (Bleq*) edge (Dleq*)
            (Bneq*) edge (Din*)
            (Dneq*) edge (Dleq*)
            (Dleq*) edge (Din*)
            (Din*) edge (all)
            ;
       \draw[thick,dashed,OliveGreen] (-.5,1.5) -- (5.75,1.5);
       \draw[OliveGreen] (-.5,-.5) rectangle (5.75,2.5);
       \draw[draw=none,fill=yellow,fill opacity=.15] (-.5,-.5) rectangle (5.75,1.5);

       \end{scope}
       
              \begin{scope}[shift={(7, 8)}]
  
     \draw (0,0) node (empty) {\tiny{$\emptyset$}}
           (1,0) node (Bin*) {\tiny{$\mathcal B(\in^*)$}}
           (1,1) node (Bleq*) {\tiny{$\mathcal B(\leq^*)$}}
           (1,2) node (Bneq*) {\tiny{$\mathcal B(\neq^*)$}}
           (3,0) node (Dneq*) {\tiny{$\mathcal D(\neq^*)$}}
           (3,1) node (Dleq*) {\tiny{$\mathcal D(\leq^*)$}}
           (3,2) node (Din*) {\tiny{$\mathcal D(\in^*)$}}
           (5,2) node (all) {\tiny{$\omega^\omega\setminus(\omega^\omega)^W$}}
           (3, -1) node (label) {\tiny{f) All nodes below $\mathcal D(\leq^*)$ empty and $\mathcal B(\neq^*)$ empty}}
           ;
     \draw[->,>=stealth]
            (empty) edge (Bin*)
            (Bin*) edge (Bleq*)
            (Bleq*) edge (Bneq*)
            (Bin*) edge (Dneq*)
            (Bleq*) edge (Dleq*)
            (Bneq*) edge (Din*)
            (Dneq*) edge (Dleq*)
            (Dleq*) edge (Din*)
            (Din*) edge (all)
            ;
       \draw[thick,dashed,OliveGreen] (2., .5) -- (2,2.5);
       \draw[thick,dashed,OliveGreen] (2., .5) -- (5.75,.5);
       \draw[OliveGreen] (-.5,-.5) rectangle (5.75,2.5);
       \draw[draw=none,fill=yellow,fill opacity=.15] (-.5,-.5) rectangle (2,2.5);
       \draw[draw=none,fill=yellow,fill opacity=.15] (2,-.5) rectangle (5.75,.5);

       \end{scope}
       
              \begin{scope}[shift={(0, 8)}]
  
     \draw (0,0) node (empty) {\tiny{$\emptyset$}}
           (1,0) node (Bin*) {\tiny{$\mathcal B(\in^*)$}}
           (1,1) node (Bleq*) {\tiny{$\mathcal B(\leq^*)$}}
           (1,2) node (Bneq*) {\tiny{$\mathcal B(\neq^*)$}}
           (3,0) node (Dneq*) {\tiny{$\mathcal D(\neq^*)$}}
           (3,1) node (Dleq*) {\tiny{$\mathcal D(\leq^*)$}}
           (3,2) node (Din*) {\tiny{$\mathcal D(\in^*)$}}
           (5,2) node (all) {\tiny{$\omega^\omega\setminus(\omega^\omega)^W$}}
           (3, -1) node (label) {\tiny{e) All nodes below $\mathcal D(\leq^*)$ empty and $\mathcal B(\neq^*)$ non empty}}
           ;
     \draw[->,>=stealth]
            (empty) edge (Bin*)
            (Bin*) edge (Bleq*)
            (Bleq*) edge (Bneq*)
            (Bin*) edge (Dneq*)
            (Bleq*) edge (Dleq*)
            (Bneq*) edge (Din*)
            (Dneq*) edge (Dleq*)
            (Dleq*) edge (Din*)
            (Din*) edge (all)
            ;
       \draw[thick,dashed,OliveGreen] (2., 1.5) -- (2,.5);
       \draw[thick,dashed,OliveGreen] (-.5, 1.5) -- (2,1.5);
       \draw[thick,dashed,OliveGreen] (2., .5) -- (5.75,.5);
       \draw[OliveGreen] (-.5,-.5) rectangle (5.75,2.5);
       \draw[draw=none,fill=yellow,fill opacity=.15] (-.5,-.5) rectangle (2.0,1.5);
       \draw[draw=none,fill=yellow,fill opacity=.15] (2.0,-.5) rectangle (5.75,.5);

       \end{scope}
       
              \begin{scope}[shift={(0, 4)}]
  
     \draw (0,0) node (empty) {\tiny{$\emptyset$}}
           (1,0) node (Bin*) {\tiny{$\mathcal B(\in^*)$}}
           (1,1) node (Bleq*) {\tiny{$\mathcal B(\leq^*)$}}
           (1,2) node (Bneq*) {\tiny{$\mathcal B(\neq^*)$}}
           (3,0) node (Dneq*) {\tiny{$\mathcal D(\neq^*)$}}
           (3,1) node (Dleq*) {\tiny{$\mathcal D(\leq^*)$}}
           (3,2) node (Din*) {\tiny{$\mathcal D(\in^*)$}}
           (5,2) node (all) {\tiny{$\omega^\omega\setminus(\omega^\omega)^W$}}
            (3, -1) node (label) {\tiny{g) All nodes below $\mathcal D(\neq^*)$ empty and $\mathcal B(\neq^*)$ empty}}
           ;
     \draw[->,>=stealth]
            (empty) edge (Bin*)
            (Bin*) edge (Bleq*)
            (Bleq*) edge (Bneq*)
            (Bin*) edge (Dneq*)
            (Bleq*) edge (Dleq*)
            (Bneq*) edge (Din*)
            (Dneq*) edge (Dleq*)
            (Dleq*) edge (Din*)
            (Din*) edge (all)
            ;
       \draw[thick,dashed,OliveGreen] (2,-.5) -- (2,2.5);
       \draw[OliveGreen] (-.5,-.5) rectangle (5.75,2.5);
       \draw[draw=none,fill=yellow,fill opacity=.15] (-.5,-.5) rectangle (2,2.5);

       \end{scope}
       
              \begin{scope}[shift={(7, 12)}]
  
     \draw (0,0) node (empty) {\tiny{$\emptyset$}}
           (1,0) node (Bin*) {\tiny{$\mathcal B(\in^*)$}}
           (1,1) node (Bleq*) {\tiny{$\mathcal B(\leq^*)$}}
           (1,2) node (Bneq*) {\tiny{$\mathcal B(\neq^*)$}}
           (3,0) node (Dneq*) {\tiny{$\mathcal D(\neq^*)$}}
           (3,1) node (Dleq*) {\tiny{$\mathcal D(\leq^*)$}}
           (3,2) node (Din*) {\tiny{$\mathcal D(\in^*)$}}
           (5,2) node (all) {\tiny{$\omega^\omega\setminus(\omega^\omega)^W$}}
            (3, -1) node (label) {\tiny{d) All nodes below $\mathcal B(\neq^*)$ empty and $\mathcal D(\neq^*)$ non empty}}
           ;
     \draw[->,>=stealth]
            (empty) edge (Bin*)
            (Bin*) edge (Bleq*)
            (Bleq*) edge (Bneq*)
            (Bin*) edge (Dneq*)
            (Bleq*) edge (Dleq*)
            (Bneq*) edge (Din*)
            (Dneq*) edge (Dleq*)
            (Dleq*) edge (Din*)
            (Din*) edge (all)
            ;
       \draw[thick,dashed,OliveGreen] (2,-.5) -- (2,1.5);
       \draw[thick,dashed,OliveGreen] (-.5,1.5) -- (2,1.5);
       \draw[OliveGreen] (-.5,-.5) rectangle (5.75,2.5);
       \draw[draw=none,fill=yellow,fill opacity=.15] (-.5,-.5) rectangle (2,1.5);

       \end{scope}
       
              \begin{scope}[shift={(0, 12)}]
  
     \draw (0,0) node (empty) {\tiny{$\emptyset$}}
           (1,0) node (Bin*) {\tiny{$\mathcal B(\in^*)$}}
           (1,1) node (Bleq*) {\tiny{$\mathcal B(\leq^*)$}}
           (1,2) node (Bneq*) {\tiny{$\mathcal B(\neq^*)$}}
           (3,0) node (Dneq*) {\tiny{$\mathcal D(\neq^*)$}}
           (3,1) node (Dleq*) {\tiny{$\mathcal D(\leq^*)$}}
           (3,2) node (Din*) {\tiny{$\mathcal D(\in^*)$}}
           (5,2) node (all) {\tiny{$\omega^\omega\setminus(\omega^\omega)^W$}}
            (3, -1) node (label) {\tiny{c) All nodes below $\mathcal B(\leq^*)$ empty, $\mathcal D(\neq^*)$ empty}}
           ;
     \draw[->,>=stealth]
            (empty) edge (Bin*)
            (Bin*) edge (Bleq*)
            (Bleq*) edge (Bneq*)
            (Bin*) edge (Dneq*)
            (Bleq*) edge (Dleq*)
            (Bneq*) edge (Din*)
            (Dneq*) edge (Dleq*)
            (Dleq*) edge (Din*)
            (Din*) edge (all)
            ;
       \draw[thick,dashed,OliveGreen] (-.5,.5) -- (5.75, .5);
       \draw[OliveGreen] (-.5,-.5) rectangle (5.75,2.5);
       \draw[draw=none,fill=yellow,fill opacity=.15] (-.5,-.5) rectangle (5.75,.5);

       \end{scope}
       
              \begin{scope}[shift={(7, 16)}]
  
     \draw (0,0) node (empty) {\tiny{$\emptyset$}}
           (1,0) node (Bin*) {\tiny{$\mathcal B(\in^*)$}}
           (1,1) node (Bleq*) {\tiny{$\mathcal B(\leq^*)$}}
           (1,2) node (Bneq*) {\tiny{$\mathcal B(\neq^*)$}}
           (3,0) node (Dneq*) {\tiny{$\mathcal D(\neq^*)$}}
           (3,1) node (Dleq*) {\tiny{$\mathcal D(\leq^*)$}}
           (3,2) node (Din*) {\tiny{$\mathcal D(\in^*)$}}
           (5,2) node (all) {\tiny{$\omega^\omega\setminus(\omega^\omega)^W$}}
            (3, -1) node (label) {\tiny{b) All nodes except $\mathcal B(\in^*)$ non empty}}
           ;
     \draw[->,>=stealth]
            (empty) edge (Bin*)
            (Bin*) edge (Bleq*)
            (Bleq*) edge (Bneq*)
            (Bin*) edge (Dneq*)
            (Bleq*) edge (Dleq*)
            (Bneq*) edge (Din*)
            (Dneq*) edge (Dleq*)
            (Dleq*) edge (Din*)
            (Din*) edge (all)
            ;
       \draw[thick,dashed,OliveGreen] (-.5,.5) -- (1.5,.5);
       \draw[thick,dashed,OliveGreen] (1.5,.5) -- (1.5,-.5);
       \draw[OliveGreen] (-.5,-.5) rectangle (5.75,2.5);
       \draw[draw=none,fill=yellow,fill opacity=.15] (-.5,-.5) rectangle (1.5,.5);

       \end{scope}
      
  \end{tikzpicture}
  \caption{All Possible Cuts in the $\leq_W$ Cicho\'n Diagram. Each one can be achieved by a proper forcing over $W$. White means that the node is not empty while yellow means that it is. No distinction is made between different non-empty nodes. Note that the trivial cut where all nodes remain empty is not shown.}
  \label{allcuts}
\end{figure}

\begin{proof}
There are two cuts I have yet to explicitly show. These correspond to e) and i) in Figure \ref{allcuts} below. 

 
 
 
 
\noindent e) All nodes below $\mathcal D(\leq^*)$ are empty and $\mathcal B(\neq^*)$ is non empty: This is the first case where we still have to prove something. Let $\mathbb P = \mathbb B * \dot{\mathbb{PT}}$. I claim that in $W^\mathbb P$ this cut is realized. We have seen that forcing with $\mathbb B$ adds an eventually different real and, by further forcing with $\mathbb{PT}$ over $W^{\mathbb B}$ will add a real which is unbounded by $W^{\mathbb B} \cap \omega^\omega$ and hence $W \cap \omega^\omega$. It remains therefore to see that in $W^\mathbb P$ there are no dominating or infinitely often equal reals over $W$. To show that there are no dominating reals, note that in general $\mathbb{PT}$ adds no dominating real, so in $W^\mathbb P$ there is no real which is dominating over $W^\mathbb B$. But, since $\mathbb B$ is $\omega^\omega$-bounding, it follows that there is no real dominating over $W$ in $W^\mathbb P$. To show there are no infinitely often equal reals, let us first note the following fact.
\begin{fact}[Corollary 2.5.2 of \cite{BarJu95}]
Suppose $M \models ZFC$. Then $M \cap 2^\omega \in \mathcal N$ if and only if there is a sequence $\langle F_n \subseteq 2^n \; | \; n  < \omega \rangle$ such that $\Sigma^\infty_{n=0} |F_n| 2^{-n} < \infty$ and for every $x \in M \cap 2^\omega$ there are infinitely many $n$ so that $x \upharpoonright n \in F_n$.
\end{fact}

As a corollary of this Fact, notice that adding an infinitely often equal real on $\omega^\omega$ makes the ground model reals measure $0$. To see why, suppose $g \in \omega^\omega$ is infinitely often equal over an inner model $M$ and let $\langle \tau_k \; | \; k < \omega \rangle$ be an enumeration in $M$ of the elements of $2^{< \omega}$. Then for every $x \in 2^\omega \cap M$ let $\hat{x} : \omega \to \omega$ be defined by $\hat{x} (n) =k$ if $x \upharpoonright n = k$. Clearly if $x \in M$ the $\hat{x} \in M$ so there are infinitely many $n$ such that $\hat{x} (n) = g(n)$. But then, pulling back, let $g ' : \omega \to 2^{< \omega}$ be defined by $g ' (n) = \sigma_k$ if $g(n) = k$ and $\sigma_k \in 2^n$ and is trivial otherwise. Then we have that for every $x \in M \cap 2^\omega$ if $\hat{x} (n) = g(n)$ then $x \upharpoonright n = g ' (n)$ so the sequence $\langle \{g ' (n) \} \; | \; n < \omega \rangle$ witnesses that $2^\omega \cap M$ is measure $0$ by the Fact. 

From this it follows immediately that $\mathbb P$ does not add infinitely often equal reals since both $\mathbb B$ (\cite[Lemma 6.3.12]{BarJu95}) and $\mathbb{PT}$ (\cite[Theorem 7.3.47]{BarJu95}) preserve outer measure.




\noindent i) All nodes below $\mathcal D(\in^*)$ are empty: This is accomplished by forcing with the infinitely often equal forcing $\mathbb{EE}$ as defined in \cite[Definition 7.4.11]{BarJu95}. The relevant facts to see that this forcing works can be deduced from \cite[Lemma 7.4.14]{BarJu95}). The details are left to the interested reader. 



\end{proof}

\section{Achieving a Full Separation in the $\leq_W$-Cicho\'n Diagram and the axiom $CD(\leq_W)$}
In this section building off the work done in the last section I build a model where there is complete separation between all elements in the diagram.
\begin{theorem}(GBC)
Given any transitive inner model $W$ of ZFC, there is a proper forcing notion $\mathbb P$, such that in $W^\mathbb P$ all the nodes in the $\leq_W$-Cicho\'n diagram are distinct and every possibile separation is simultaneously realized. 
\label{ConFS}
\end{theorem}

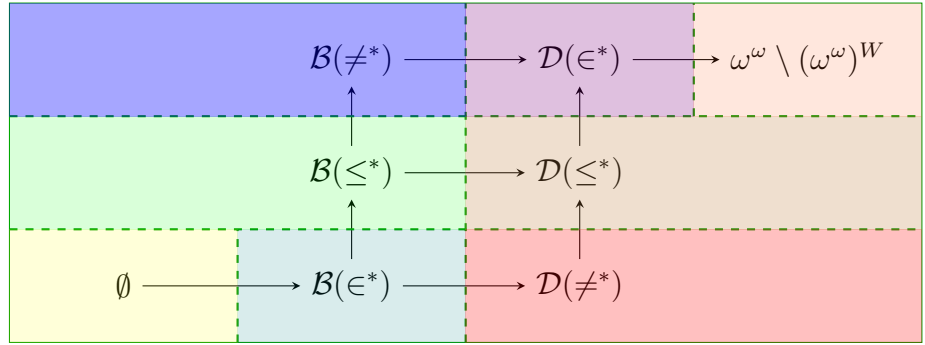
\begin{figure}[h]\label{Figure.Cichon-basic}
\centering
  \begin{tikzpicture}[scale=1.5,xscale=2]
     \draw (0,0) node (empty) {$\emptyset$}
           (1,0) node (Bin*) {$\mathcal B(\in^*)$}
           (1,1) node (Bleq*) {$\mathcal B(\leq^*)$}
           (1,2) node (Bneq*) {$\mathcal B(\neq^*)$}
           (2,0) node (Dneq*) {$\mathcal D(\neq^*)$}
           (2,1) node (Dleq*) {$\mathcal D(\leq^*)$}
           (2,2) node (Din*) {$\mathcal D(\in^*)$}
           (3,2) node (all) {$\omega^\omega\setminus(\omega^\omega)^W$}
           ;
     \draw[->,>=stealth]
            (empty) edge (Bin*)
            (Bin*) edge (Bleq*)
            (Bleq*) edge (Bneq*)
            (Bin*) edge (Dneq*)
            (Bleq*) edge (Dleq*)
            (Bneq*) edge (Din*)
            (Dneq*) edge (Dleq*)
            (Dleq*) edge (Din*)
            (Din*) edge (all)
            ;
       \draw[thick,dashed,OliveGreen] (-.5,.5) -- (3.5,.5);
       \draw[thick,dashed,OliveGreen] (1.5,-.5) -- (1.5,2.5);
       \draw[thick,dashed,OliveGreen] (-.5,1.5) -- (3.5,1.5);
       \draw[thick,dashed,OliveGreen] (.5,-.5) -- (.5,.5);
       \draw[thick,dashed,OliveGreen] (2.5,1.5) -- (2.5,2.5);
       \draw[OliveGreen] (-.5,-.5) rectangle (3.5,2.5);
       \draw[draw=none,fill=yellow, fill opacity=.15] (-.5,-.5) rectangle (.5,.5);
       \draw[draw=none,fill=teal,fill opacity=.15] (.5,-.5) rectangle (1.5,.5);
       \draw[draw=none,fill=red,fill opacity=.25] (1.5,-.5) rectangle (3.5,.5);
       \draw[draw=none,fill=green,fill opacity=.15] (-.5,.5) rectangle (1.5,1.5);
       \draw[draw=none,fill=brown,fill opacity=.25] (1.5,.5) rectangle (3.5,1.5);
       \draw[draw=none,fill=blue,fill opacity=.35] (-.5,1.5) rectangle (1.5,2.5);
       \draw[draw=none,fill=violet,fill opacity=.25] (1.5,1.5) rectangle (2.5,2.5);
       \draw[draw=none,fill=Orange,fill opacity=.15] (2.5,1.5) rectangle (3.5,2.5);
  \end{tikzpicture}
  \caption{Full Separation of the $\leq_W$-diagram}
\end{figure}

\noindent In what follows I call the axiom \say{All consistent separations of the $\leq_W$-diagram are distinct} $CD(\leq_W)$ or \say{full Cicho\'n Diagram for $\leq_W$}. Thus the above theorem states that $CD(\leq_W)$ can be forced over $W$ by a proper forcing. For different inner models $W$ the sentence $CD(\leq_W)$ may vary but they can all be forced the same way.

Before proving this theorem I need a simple technical result about Sacks and Laver forcing.

\begin{lemma}
The product forcing $\mathbb S \times \mathbb L$ satisfies Axiom A and hence is proper.
\label{sackslaver}
\end{lemma}

\begin{proof}
Theorem 1 of \cite{Gro87} gives a general framework for showing that certain arboreal forcings satisfy Axiom A (including Sacks and Laver forcings) and here I adapt the proof to the case of a product of two arboreal forcings. Recall that if $p, q \in \mathbb S$ and $n \in \omega$ then we let $q \leq^\mathbb S_n p$ if and only if $q \subseteq p$ and every $n^{\rm th}$ splitting node of $q$ is an $n^{\rm th}$ splitting node of $p$ i.e. if $\tau \in q$ is a splitting node with $n$ splitting predecessors in $q$ then the same is true of $\tau$ in $p$. Also, given a canonical enumeration of $\omega^{<\omega}$ in which $\sigma$ appears before $\tau$ if $\sigma \subseteq \tau$ and $\sigma^\frown k$ appears before $\sigma^\frown(k+1)$ then for $p \in \mathbb L$ one gets an enumeration of the elements of $p$ above the stem, $\sigma_1^p,...,\sigma_k^p,...$ and if $p, q \in \mathbb L$ and $n \in \omega$ then let $q \leq^\mathbb L_n p$ if and only if $q \subseteq p$ and $s_i^p = s_i^q$ for all $i =0,...,n$. Clearly if for every $n \in \omega$ and $(p_s, p_l) , (q_s, q_l) \in \mathbb S \times \mathbb L$ we let $(q_s, q_l) \leq_n (p_s, p_l)$ if and only if $q_s \leq_n^\mathbb S p_s$ and $q_l \leq_n^\mathbb L s_l$ then this satisfies the first requirement of Axiom A forcings. Thus, it remains to show that for every $\mathbb S \times \mathbb L$-name $\dot{a}$ and condition $(p_s, p_l) \in \mathbb S \times \mathbb L$ if $(p_s, p_l) \Vdash \dot{a} \in \check{V}$ then for every $n$ there is a $(q_s, q_l)$ and a countable set $A \in V$ such that $(q_s, q_l) \Vdash \dot{a} \in A$.

Fix such a name $\dot{a}$ and condition $p = (p_s, p_l)$. Let $D \subseteq \mathbb S \times \mathbb L$ be the set of all $(q_s, q_l) \leq p$ such that there is some $a(q) \in V$ with $(q_s, q_l) \Vdash \check{a(q)} = \dot{a}$. This set is dense below $p$ since $p$ forces $\dot{a}$ to be an element of $V$. Let $H_D \subseteq p$ be the set of all pairs $(\sigma, \tau) \in p$ such that there is a $(\sigma ', \tau') \subseteq (\sigma, \tau)$ with $\sigma '$ $n$-splitting in $p_s$ and $\tau'$ $n$-splitting in $p_l$ and there is some $r_{\sigma, \tau} = (r_s, r_l) \leq p$ in $D$ whose stem (i.e. the pair of the stems from the two components) is $(\sigma ,  \tau )$. Finally let $Min(H_D)$ be the set of $(\sigma, \tau) \in H_D$ which are minimal with respect to inclusion. Note that $Min(H_D)$ is an antichain since no two elements can be comparable and both minimal. Let $r = (r_s, r_l) = \bigcup \{r_{\sigma, \tau} \; | \; (\sigma, \tau) \in Min(H_D)\}$. A routine check shows that the set $r$ is a condition in $\mathbb S \times \mathbb L$ and $r \leq_n p$.

Now let $A = \{a(r_{\tau, \sigma}) \; | \; (\sigma, \tau) \in Min(H_D)\}$. This set is countable thus to finish the lemma it suffices to show that $r \Vdash \dot{a} \in \check{A}$. To see this, suppose that $t \leq r$ and $t \Vdash \dot{a} = \check{a}$ for some $a$. By extending $t$ if necessary one may assume that the stem of $t$ is in $H_D$. But then some initial segment of the stem is in $Min(H_D)$ so $a \in A$, as needed.
\end{proof}

Now I prove Theorem \ref{ConFS}.
\begin{proof}[Proof of Theorem \ref{ConFS}]
This essentially follows from the theorems of the previous section. Given a definable forcing notion $\mathbb Q$ let me write $\mathbb Q^W$ for the version of that forcing notion as computed in $W$. Let $\mathbb P = \mathbb S^W \times \mathbb L^W \times \mathbb{LOC}^W$. Then in $W^\mathbb P$ not every new real is in an element of the diagram since Sacks reals were added. Moreover, by our arguments above the combination of $\mathbb{LOC}$ and $\mathbb{L}$ will add reals to every node of the diagram but, none of them will be equal and moreover every possible non-separation is realized as one observes by my previous arguments. 

It remains to see that $\mathbb P$ is proper. This follows from Lemma \ref{sackslaver} plus the fact that $\mathbb{LOC}$ is $\sigma$-linked and hence indestructibly ccc.
\end{proof}

Let me finish this paper by briefly studying the axiom $CD(\leq_W)$. First, let me show that there are other ways to obtain it. Indeed there is another, less finegrained approach to forcing $CD(\leq_W)$. To describe this, let me make the following simple observation. Recall that the {\em Maximality Principle} $MP$ of \cite{Hamkins03} states that any statement which is forceably necessary or can be forced to be true in such as a way that it cannot become later forced to be false, is already true. If $\Gamma$ is a class of forcings then the maximality principle for $\Gamma$, $MP_\Gamma$, states the same but only with respect to forcings in $\Gamma$.
\begin{proposition}
The axiom $CD(\leq_W)$ is forceably necessary, that is once it has been forced to be true it will remain so in any further forcing extension. Thus in particular it is implied by the maximality principle, $MP$.
\label{FSFN}
\end{proposition}

\begin{proof}
This is more or less immediate from the definition. Since $CD(\leq_W)$ is defined relative to a fixed inner model and the diagram for $W$ concerns only the models $W[x]$ for $x \in \omega^\omega \cap V$, notice that forcing over $V$ cannot change the theories of the models $W[x]$ for $x \in V$ hence if $CD(\leq_W)$ is true in $V$ it must remain so in any forcing extension. In other words absoluteness for membership in each of the various classes holds and this guarentees that forcing cannot change the relation $x \in A$ for any node $A$ of the diagram.

Since $CD(\leq_W)$ is forceably necessary it follows that $MP$ implies $CD(\leq_W)$.
\end{proof}

Now notice that since all the forcing notions used in Theorem \ref{ConFS} have size at most $2^{\aleph_0}$ it follows that the collapse forcing $Coll(\omega, < (2^{2^{\aleph_0}})^+)$ will add a generic making $CD(\leq_W)$ true. Since $CD(\leq_W)$ is forceably necessary it follows that the full collapse forcing cannot kill the generic once it is added and, as a result one obtains
\begin{corollary}
$W^{Coll(\omega, < (2^{2^{\aleph_0}})^+)} \models CD(\leq_W)$
\end{corollary}
Moreover, note that while the forcing described in Theorem \ref{ConFS} was proper and hence preserved $\omega_1$ the collapse forcing used above is not. Therefore the following is immediate.
\begin{corollary}
The statement \say{the reals of $W$ are countable} is independent of the theory ZFC($W$) + $CD(\leq_W)$.
\end{corollary}

Since $CD(\leq_W)$ is forceably necessary and hence cannot be killed once it is forced to be true it follows that any sentence which can be forced to be true from any model must be consistent with $CD(\leq_W)$. Such examples include $CH$, $2^{\aleph_0} = \kappa$ for any $\kappa$ of uncountable cofinality, Martin's Axiom and its negation, $\diamondsuit$ and its negation, and a wide variety of forcings associated with the classical Cicho\'n's diagram. In particular, $CD(\leq_W)$ is independent of any consistent assignment of cardinals to the nodes in the Cicho\'n diagram (cf \cite{BarJu95} for a variety of examples of such).

Let me finish now by showing the consistency of a strong version of $CD(\leq_W)$, which was suggested to me by Gunter Fuchs. The idea is to iteratively force with the forcing $\mathbb P$ of Theorem \ref{ConFS} for long enough that a large collection of inner models $W$ simultaneously satisfy $CD(\leq_W)$.
\begin{theorem}
Assume $V=L$. Then there is an $\aleph_2$-c.c. proper forcing extension where $2^{\aleph_0} = \aleph_2$ and for every $\aleph_1$-sized set of reals $A$ there is a set of reals $B \supseteq A$ of size $\aleph_1$ so that $CD(\leq_W)$ holds for $W = L[B]$.
\label{ConIFS}
\end{theorem}

\begin{proof}
Assume $V=L$ and let $\vec{\mathbb P} = \langle (\mathbb P_\alpha, \dot{\mathbb Q}_\alpha) \;  | \; \alpha < \omega_2 \rangle$ be an $\omega_2$-length countable support iteration of copies of the forcing $\mathbb P$ from Theorem \ref{ConFS} (i.e. $\dot{\mathbb Q}_{\alpha +1}$  evaluates to $( \mathbb P)^{L^{\mathbb P_\alpha}}$). Clearly $\vec{\mathbb P}$ is proper. Moreover, since CH holds in the ground model and the forcing $\mathbb P$ is easily seen to be of size continuum, and does not kill CH it follows that $\vec{\mathbb P}$ has the $\aleph_2$-c.c. and every intermediate stage in the iteration preserves CH: $L^{\mathbb P_\alpha} \models {\rm CH}$ for all $\alpha < \omega_2$. However, since reals are added at every stage the final model satisfies $2^{\aleph_0} = \aleph_2$.

It remains to show that for every $\aleph_1$-sized set of reals $A$ there is a set of reals $B \supseteq A$ of size $\aleph_1$ so that $CD(\leq_W)$ holds for $W = L[B]$. Let $A$ be a set of reals of size $\aleph_1$. Then, there is some $\alpha$ so that $A \in L[G_\alpha]$ for $G_\alpha$ be $\mathbb P_\alpha$-generic. Note that we can code $G_\alpha$ by a set of reals of size at most $\aleph_1$, say $B$, and without loss we can assume that $A \subseteq B$ for $L[G_\alpha] = L[B]$. Then at stage $\mathbb P_{\alpha+1}$ we added a generic witnessing that $CD(\leq_{L[B]})$ holds. Moreover, by the fact that this statement is forceably necessary, it cannot be killed by the tail end of the iteration so it holds in the final model.
\end{proof}

While it is not entirely clear what consequences we can expect from $CD(\leq_W)$ for an arbitrary $W$, the stronger version obtained in Theorem \ref{ConIFS} has several low hanging fruits in this regard. Let me pluck a particularly simple one connecting the constructibility diagram to the standard Cicho\'n diagram. 
\begin{lemma}
Assume for every $\aleph_1$-sized set of reals $A$ there is a set of reals $B \supseteq A$ of size $\aleph_1$ so that $CD(\leq_W)$ holds for $W = L[B]$. Then all the cardinals in the Cicho\'n diagram have size at least $\aleph_2$.
\end{lemma}

\begin{proof}
It suffices to show that ${\rm add}(\mathcal N) \geq \aleph_2$. Towards this goal, recall Bartoszy\'nski's characterization of ${\rm add}(\mathcal N)$ as the least cardinal $\kappa$ so that there is a set of reals $X$ of size $\kappa$ so that no single slalom can capture all the reals in $X$ (\cite[Theorem 5.14]{BlassHB}). The result is then immediate for, given any set of reals $A$ of size $\aleph_1$, we can find a set $B \supseteq A$ of size $\aleph_1$ and a slalom $\sigma$ eventually capturing all reals in $L[B]$ by $CD(\leq_{L[B]})$ so ${\rm add}(\mathcal N) > \aleph_1$.
\end{proof}



\section{Open Questions}
I finish by collecting the open questions that have appeared throughout this paper. First I ask about the Cicho\'n diagram for other reduction concepts. Recall that in the case of $\leq_T$, the sets $\mathcal B(\in^*)$ and $\mathcal B(\leq^*)$ were equal.
\begin{question}
For which reductions $(\sqsubseteq, x_0)$ on the reals is $\mathcal B_\sqsubseteq(\in^*) \subsetneq \mathcal B_\sqsubseteq(\leq^*)$?
\end{question}
The anonymous referee has pointed out to me that Monin (unpublished) has shown that for hyperarithmetic reductions $\mathcal B_\sqsubseteq(\in^*) = \mathcal B_\sqsubseteq(\leq^*)$. See \cite[Fact 2.6]{Kihara17}. This question has also been considered in \cite{Kihara17}, see Problem 5.7 and the discussion preceeding it. This shows that for many natural reduction concepts the answer to the question above is negative. Taking this into account it seems reasonable to ask if indeed {\em any} ``reasonable" reduction concept (whatever that means) provably does this in $\mathsf{ZFC}$? Note that if $V=L$ all $\leq_W$ relations are trivial.

Next I ask about the ZFC($W$)-provable relations between the nodes of the $\leq_W$-Cicho\'n diagram. While I have shown that there are no other implications it is entirely possible that there are other relations more generally.
\begin{question}
What other ZFC($W$)-provable relations are there between the sets in?
\end{question}

My next collection of questions concerns the subforcings of $\mathbb{LOC}$, a topic that deserves more study.
\begin{question}
What is the forcing adding the eventually different real described in Lemma \ref{evdiffloc}? Does it add a dominating real? Note that it must be ccc, in fact $\sigma$-linked and add eventually different reals which are bounded by nearly all ground model reals.
\end{question}

Similarly, one might ask whether there is a similarly exotic subforcing of $\mathbb{LOC}$ for adding a dominating real.
\begin{question}
Does every subforcing of $\mathbb{LOC}$ adding a dominating real add a $\mathbb D$-generic real?
\end{question}
\begin{question}
Does every subforcing of $\mathbb{LOC}$ add a Cohen real or a random real?
\end{question}

Finally I conclude with some questions about the axiom $CD(\leq_W)$.
\begin{question}
What statements are implied by $CD(\leq_W)$? In particular, does it imply that there are $W$-generics for the forcings to add reals we have discussed (Cohen, random, etc)?
\end{question}
\begin{question}
How does $CD(\leq_W)$ relate to standard forcing axioms? In particular does ${\rm MA}_{\aleph_1}$ imply $CD(\leq_{L[A]})$ for all $\aleph_1$-sized sets of reals $A$? Does BPFA?
\end{question}

\bibliography{Logicpaperrefs}

\begin{thebibliography}{10}

\bibitem{Bar1987}
Tomek Bartoszy\'nski.
\newblock Combinatorial aspects of measure and category.
\newblock {\em Fundamenta Mathematicae}, 127(3):225--239, 1987.

\bibitem{BarJu95}
Tomek Bartoszy\'nski and Haim Judah.
\newblock {\em Set Theory: On the Structure of the Real Line}.
\newblock A.K. Peters, Wellsley, MA, 1995.

\bibitem{BlassHB}
Andreas Blass.
\newblock Combinatorial cardinal characteristics of the continuum.
\newblock In Matthew Foreman and Akihiro Kanamori, editors, {\em Handbook of
  Set Theory}, pages 395--489. Springer, Dordrect, 2010.

\bibitem{BBNN}
J\"org Brendle, Andrew Brooke-Taylor, Keng~Meng Ng, and Andr\'e Nies.
\newblock An analogy between cardinal characteristics and highness properties
  of oracles.
\newblock In Xishun Zhao, Qi~Feng, Byunghan Kim, and Liang Yu, editors, {\em
  Procedings of the 13th Asain Logic Conference}, pages 1--29. World
  Scientific, Singapore, 2015.

\bibitem{BL11}
J\"org Brendle and Benedikt L\"owe.
\newblock Eventually different functions and inaccessible cardinals.
\newblock {\em Journal of the Mathematical Society of Japan}, 63(1):137--151,
  2011.

\bibitem{GS93}
Moti Gitik and Saharon Shelah.
\newblock More on simple forcing notions and forcings with ideals.
\newblock {\em Annals of Pure and Applied Logic}, 59:219--238, 1993.

\bibitem{GreenbergTuretskyKuyper}
Noam Greenberg, Rutger Kuyper, and Dan Turetsky.
\newblock Cardinal invariants, non-lowness classes, and {W}eihrauch
  reducibility.
\newblock {\em Computability}, 8(3-4):305--346, 2019.

\bibitem{Gro87}
Marcia~J. Groszek.
\newblock Combinatorics on ideals and forcing with trees.
\newblock {\em The Journal of Symbolic Logic}, 52(3):582--593, September 1987.

\bibitem{Hamkins03}
Joel~David Hamkins.
\newblock A simple maximality principle.
\newblock {\em The Journal of Symbolic Logic}, 68(2):527--550, 2003.

\bibitem{ShJdKM}
Haim Judah and Saharon Shelah.
\newblock The {K}unen-{M}iller chart ({L}ebesgue measure, the {B}aire property,
  {L}aver reals and preservation theorems for forcing).
\newblock {\em The Journal of Symbolic Logic}, pages 909--927, 1990.

\bibitem{Kihara17}
Takayuki Kihara.
\newblock Higher randomness and lim-sup forcing within and beyond
  hyperarithmetic.
\newblock In {\em Sets and computations}, volume~33 of {\em Lect. Notes Ser.
  Inst. Math. Sci. Natl. Univ. Singap.}, pages 117--155. World Sci. Publ.,
  Hackensack, NJ, 2018.

\bibitem{Pal13}
Justin~Thomas Palumbo.
\newblock Hechler forcing and its relatives.
\newblock {\em PhD Thesis, UCLA}, 2013.

\bibitem{Rupprecht}
Nicholas~Andrew Rupprecht.
\newblock {\em Effective correspondents to cardinal characteristics in
  {C}ichon's diagram}.
\newblock ProQuest LLC, Ann Arbor, MI, 2010.
\newblock Thesis (Ph.D.)--University of Michigan.

\bibitem{Switz18b}
Corey~Bacal Switzer.
\newblock Invariants of measure and category as degrees of constructibility.
\newblock In Preparation, 2018.

\bibitem{dimtheoryandforcing}
Jind\v{r}ich Zapletal.
\newblock Dimension theory and forcing.
\newblock {\em Topology Appl.}, 167:31--35, 2014.

\end{thebibliography}
\bibliographystyle{plain}

\end{document}